\documentclass[12pt,a4paper]{amsart}
\usepackage[top=32mm, bottom=32mm, left=32mm, right=32mm]{geometry}
\usepackage{mathptmx}
\usepackage{mathrsfs}
\usepackage{amscd}
\usepackage{verbatim}
\usepackage{url}
\usepackage[all]{xy}
\usepackage{color}
\usepackage[colorlinks=true,citecolor=blue]{hyperref}
\usepackage{tikz}
\usepackage{amsfonts}
\usepackage{bbm}
\usepackage{graphicx}
\usepackage{amsmath}
\usepackage{amssymb}
\usetikzlibrary{arrows}
\usetikzlibrary{patterns}
\usetikzlibrary{lindenmayersystems,backgrounds}
\theoremstyle{plain}

\newtheorem{thm}{Theorem}[section]
\newtheorem{lem}[thm]{Lemma}
\newtheorem{prop}[thm]{Proposition}
\newtheorem{cor}[thm]{Corollary}
\newtheorem{example}[thm]{Example}

\theoremstyle{definition}
\newtheorem{defn}[thm]{Definition}
\newtheorem{rem}[thm]{Remark}

\newcommand{\Z}{{\mathbb{Z}}}
\newcommand{\N}{\mathbb{N}}

\def \C {\mathbb C}

\begin{document}

\title[Structures of $R(f)-\overline{P(f)}$ for graph maps $f$]{Structures of $R(f)-\overline{P(f)}$ for graph maps $f$}

\author[J.~Mai]{Jiehua Mai}
\address{School of Mathematics and Quantitative
Economics, Guangxi University of Finance and Economics, Nanning, Guangxi, 530003, P. R. China \&
 Institute of Mathematics, Shantou University, Shantou, Guangdong, 515063, P. R. China}
\email{jiehuamai@163.com; jhmai@stu.edu.cn}

\author[E.~Shi]{Enhui Shi}
\address{School of Mathematics and Sciences, Soochow University, Suzhou, Jiangsu 215006, China}
\email{ehshi@suda.edu.cn}

\author[K.~Yan]{Kesong Yan*}
\thanks{*Corresponding author}
\address{School of Mathematics and Quantitative
Economics, Guangxi University of Finance and Economics, Nanning, Guangxi, 530003, P. R. China}
\email{ksyan@mail.ustc.edu.cn}

\author[F.~Zeng]{Fanping Zeng}
\address{School of Mathematics and Quantitative
Economics, Guangxi University of Finance and Economics, Nanning, Guangxi, 530003, P. R. China}
\email{fpzeng@gxu.edu.cn}


\date{Oct. 18, 2023}

\begin{abstract}
Let $G$ be a graph and $f: G\rightarrow G$ be a continuous map.
We establish a structure theorem which describes the structures of the set $R(f)-\overline{P(f)}$, where $R(f)$ and $P(f)$ are
the recurrent point set and the periodic point set of $f$ respectively. Roughly speaking,
the set  $R(f)-\overline{P(f)}$ is covered by finitely many pairwise disjoint $f$-invariant open sets $U_{1\,},\,\cdots,\,U_{n\,}$; each $U_i$ contains a unique minimal set $W_i$ which absorbs each point of $U_i$; each $W_i$ lies in finitely many pairwise disjoint circles each
of which is contained in a connected closed set; all of these connected closed sets are contained in $U_i$ and permutated cyclically by $f$. As applications of the structure theorem, several known results are improved or
reproved.
\end{abstract}

\maketitle

%%%%%%%%%%%%%%%%%%%%%%%%%%%%%%%%%%%%%%%%%%%%%%%%%%%%%%%%%%%%%%%%%%%%%%%%%%%%%%%%%%%%%%%%%%%%%%%%%%%%%%%%%%%%%%%%%%%%%%%%%%%%%%%%%%%%%%%%%%%%%%%%%%%%%%%%%%%%%%%%%%%%%
\section{Introduction and prelimilaries}

In this section, we will state the main theorem obtained. We will also introduce the backgrounds of the study, the notions and notations used in the paper, and the organizations of the paper.

\subsection{Backgrounds and the aim of the paper}
The study of the dynamics of graph maps can date back to the work of A. Blokh in 1980's (\cite{Bl86, Bl87, Blo87}). Since then, lots of literatures appeared in this area. One may consult \cite{ALM} for a systematic introduction to the combinatorial dynamics and chaotic
phenomena for graph maps before 2000 and consult \cite{AGG, AJM, GL, HKO, LLO, LOYZ, LS, RS} and their references for later investigations.
\medskip

Recurrence is one of the most fundamental notions in the theory of dynamical system. For a compact metric space $X$ and a continuous
map $f:X\rightarrow X$, there are several $f$-invariant subsets of $X$ which exhibit various recurrence behaviors, such as the periodic point set $P(f)$, the almost periodic point set $AP(f)$, and the recurrent point set $R(f)$. The $\omega$-limit point set $\omega(f)$ and
 the nonwandering point set $\Omega(f)$ are also $f$-invariant and exhibit some weaker dynamical behaviors than recurrence. The structures of these sets have been intensively studied during the development of dynamical systems. One may consult \cite{Ak, Fu} for the introductions to the abstract theory of recurrence and its applications in number theory.
\medskip

Due to the simplicity of the phase spaces in topology, finer and more interesting results around recurrence can be obtained for
graph maps. Blokh constructed the spectral decomposition of the sets $\overline{P(f)},\ \omega (f)$ and
$\Omega(f)$ for any graph map $f$, and obtained a series of
applications of the spectral decomposition (\cite{Bl95}). In
\cite{CH} and \cite{Sar},\, the authors showed that $\overline
{R(f)}=\overline{P(f)}$ \, for any interval map $f$. This was extended by Ye
 to tree maps (\cite{Ye}). For any graph map $f$, the
authors proved that \,$\overline{R(f)}=R(f)\cup\overline{P(f)}$\;\!
and \,$\overline{R(f)}=\overline{AP(f)}$\;\! (\cite{Haw, MS2}). We
suggest the readers to refer to \cite{AHNO, Bl15, MSh, MS4} for the study of recurrence for maps on phase spaces beyond graphs.
\medskip

The aim of this paper is to study the structure of\,
\vspace{0.5mm} $R(f)-\overline{P(f)}$\, for any graph map $f:G\to
G$. We will show that $R(f)-\overline{P(f)}$ is contained in
$G-\overline{EP(f)}$ which has only finitely many connected components\, $U_{1\,},\,\cdots,\,U_{n\,}$, where $EP(f)$ is the eventually
periodic point set of $f$. Then we describe the dynamical
behavior of $f$ \;\!on the intersection of each $U_i$
with $R(f)$\;\!. In the last of this section, we will give an explicit statement of the main theorem.

%%%%%%%%%%%%%%%%%%%%%%%%%%%%%%%%%%%%%%%%%%%%%%%%%%%%%%%%%%%%%%%%%%%%%%%%%%%%%%%%%%%%%%%%%%%%%%%%%%%%%%%%%%%%%%%%%%%%%%%%%%%%%%%%%%%%%%%%%%%%%%%%%%%%%%%%%%%%%%%%%%%%%
\subsection{Notions and notations}

Let $(X,d)$ be a metric space with metric $d$. For any\;\! $Y\subset
X$, \,denote by\, $\mbox{Int}_X(Y)$, \,$\partial_XY$, and\,
$\mbox{Clos}_X(Y)$ the interior, \,the boundary, and the closure
of\, $Y$ in $X$, respectively. If there is no confusion, we also
write \,$\overline Y$ for\, $\mbox{Clos}_X(Y)$. For any $y\in
Y\!\subset X$ and any\;\! $r>0$,\;\! write\, $B(y,r)=\{x\in
X:\,d(x,y)<r\}$\, and \,$B(Y,r)=\{x\in X:\,d(x,Y)<r\}$.

\medskip By a dynamical system, we mean a pair $(X, f)$,
where $X$ is a compact metric space and $f:X\rightarrow X$ is a
continuous map. Denote by \;\!$C^{\;\!0}(X)$\;\! the set of all
continuous maps from $X$ to $X$. Let\;\! $\Bbb N$ \;\!be the set
of all positive integers, \,and let\;\! $\Bbb Z_+=\Bbb
N\cup\{0\}$. For any \;\!$n\in\Bbb N$\;\!, \,write \;\!$\Bbb
N_n=\{1,\cdots , n\}$. For any $f\in C^{\;\!0}(X)$\;\!,\, let
$f^{\;\!0}$ be the identity map of $X$, and let $f^{\;\!n}=f\circ
f^{\;\!n-1}$ be the composition map of $f$ and $f^{\;\!n-1}$.
For $x\in X$, the set \;\!$O(x,f)\equiv\{f^{\,n}(x): n\in \Bbb Z_+\}$ \,is called the
{\it{orbit}}\, of \;\!$x$\;\! under $f$. A
point $x\in X$ is called a {\it{fixed point}} \,of $f$ if
\;\!$f(x)=x$\;\!;  is called a {\it{periodic
point}}\, of $f$ if $f^{\,n}(x)=x$ \,for some $n\in\N$\;\!;  is called an {\it{eventually periodic point}} \,of
$f$ if the orbit \;\!$O(x,f)$ \;\!is a finite set;\,
is called a {\it{non-wandering point}} \,of $f$ if for any
neighborhood \,$U$ of \;\!$x$\;\! in $X$ there is an\;\! $n\in
\Bbb N$\;\! such that $f^{n}(U)\cap U\neq\,\emptyset$\;\!.
 The set\,
$\omega(x,f)\equiv\,\bigcap_{\;\!i=0\,}^{\;\!\infty}\overline{O(f^i(x),f)}$
\vspace{0.5mm}\,is called the\, $\omega$-{\it{limit set}}
of $x$ under $f$. Write\, $\omega(f)=\bigcup_{x\in
X}\omega(x,f)$\;\!, \vspace{0.5mm}\,called the\,
$\omega$-{\it{limit set}} \,of $f$. The point $x\in X$ is called a
{\it{recurrent point}} \,of $f$ if\;\! $x\in
\omega(x,f)$\;\! and is called an {\it{almost periodic
point}}\, of $f$ \,if \,for any neighborhood \,$U$ of \;\!$x$\;\!
in $X$ there exists an\;\! $m\in\,\Bbb N$ \;\!such that
$\{f^{\;\!n+i}(x):i\in\Bbb N_m\}\cap\;\!U\neq\,\emptyset$ \,for
every $n\in\Bbb Z_+$\;\!.  A subset $W$ of $X$ is said to be
{\it{$f$-invariant}} \;if
$f(W)\subset W$; is said to be {\it{strongly
$f$-invariant}} \;if $f(W)=W$; is said to be a
{\it{minimal set}} \,of $f$ \,if it is non-empty, closed
and $f$\;\!-\,invariant and if no proper subset of $W$ has these
three properties. \,A minimal set $W$ of $f$ is said to be
{\it{totally minimal}} \,if it is a minimal set of
$f^{\;\!n}$ for all $n\in \N$\;\!.\, Denote by ${\rm Fix}(f)$\;\!,
$P(f)$\;\!, $EP(f)$\;\!, $AP(f)$\;\!, $R(f)$ \,and\,
$\Omega(f)$\;\! the sets of fixed points, periodic points,
eventually periodic points, almost periodic points, recurrent
points and non-wandering points of $f$, respectively. From the
definitions it is easy to see that \,$P(f)\subset EP(f)$ \,and\,
${\rm Fix}(f)\subset P(f)\subset AP(f)\subset R(f)\subset \omega
(f)\subset \Omega (f)$\;\!.

\medskip
A non-degenerate metric space $X$ is called an
{\it{arc}} (resp.\,\;an {\it{open arc}}\,,\ \;a
{\it{circle}}) if\;\! it is homeomorphic to the interval
$[0,1]$ (\;\!resp.\, the open interval $(0,1)$\,,\ \;the unit circle
$S^1$). A compact connected metric space $G$ is called a
(topological) {\it{graph}}\, if there exists a  finite
subset $V(G)$ \,of\, $G$ such that each connected component of\,
$G-V(G)$\;\! is an open arc, \,and any circle in \,$G$ contains at
least three points in $V(G)$\;\!. Every point in the given finite
subset $V(G)$ is called a {\it{vertex}}\, of \;\!$G$. \,Every
connected component of \;\!$G-V(G)$ is called an
{\it{edge}}\, of \;\!$G$.\;\! A graph containing no circle is
called a {\it{tree}}\;\!. \;\!A continuous map from a graph
(resp.\, a tree, \,a circle,\, an interval) to itself is called a
{\it{graph map}} (resp.\,\;a {\it{tree map}}\,,\, a
{\it{circle map}}\,,\, an {\it{interval map}})\;\!.
\,Note that if $X$ is a non-degenerate connected closed subset of a
graph\;\! $G$ then $X$ itself is also a graph.

\medskip Let $G$ be a graph. We may assume that the
metric $d$ on $G$ satisfies the following two conditions\;\!: (1)
for any $x\in G$ and any\;\! $r>0$, the open ball $B(x,r)$ in $G$
is a connected set\;\!; (2)\;\! $d(u,v)\ge1$\;\!, for any two
different vertices $u$ and $v$ of \,$G$. For any finite set $S$,\,
denote by $|S|$ the number of elements of $S$.\, For any $x\in
G$,\, write\, ${\rm val}_{\;\!G}(x)=$
$\lim_{\;r\,\to\,0}|\partial_G B(x,r)|$\,,\, called the
{\it{valence}} of $x$ in $G$; \,$x$ is called a
{\it{branching point}} (resp. an {\it{endpoint}}) of
$G$ if \,${\rm val}_{\;\!G}(x)>2$ (resp. ${\rm
val}_{\;\!G}(x)=1$). Denote by\, $\mbox{Br}(G)$ \,and
\,$\mbox{End}(G)$ \,the sets of branching points and endpoints of
\,$G$, respectively. For any arc $A$ in \;\!$G$ and any two points
\,$a,b\in A$, \,denote by $[a,b]_A$ the smallest connected closed
subset of $A$ containing $a$ and $b$\;\!. If there is no
confusion, we also write $[a,b]$ for $[a,b]_A$. In addition, we
write $(a,b]=[b,a)=[a,b]-\{a\}$\, and\, $(a,b)=(a,b]-\{b\}$. Note
that $[a,a]=\{a\}$ \,and \,$(a,a]=(a,a)=\,\emptyset$\;\!.

\medskip  For any metric space (even for any topological
space)\,$X$, any $f\in C^{\;\!0}(X)$ and any $n\in\N$\;\!, by the
definitions, it is easy to see that $P(f)=P(f^{\,n})$ and
$\,\omega(f)=\,\omega(f^{\,n})$\;\!. Erd\"{o}s and Stone
proved that $R(f)=R(f^{\,n})$ and $AP(f)=AP(f^{\,n})$ also hold (\cite{ES}).
It is well known
that if $X$ is a compact metric space and $f\in C^{\;\!0}(X)$\;\!,
then a point $x\in AP(f)$ \,if and only if \;$\overline{O(x,f)}$\,
is a minimal set of $f$ (see e.g. \cite[Proposition V.5\;\!]{BC}).

\subsection{Organizations and the statement of the main theorem}

 In section 2, for a graph map $f$ on a graph $G$ and for a subset $K$ of $G$, we introduce the notions of absorbed set and main absorbed set of $K$, and use them to analyse the structures of the orbits of some specified connected subsets of $G$ under $f$. In the beginning of Section 3, we will recall the structure theorem of graph maps without periodic points obtained by Mai and Shao in \cite{MS1}, which is a key ingredient in the proof of the main theorem. Then, based on these preparations, we prove the
following main theorem (Theorem \ref{main theorem}).

\begin{thm} \label{main result}{\it Let  $G$  be a connected graph and  $f:G\rightarrow G$ be a continuous map such that $P(f)\not=\emptyset$ and $R(f)-\overline{P(f)}\not=\emptyset$.  Then there exist pairwise disjoint nonempty open subsets  $U_1, \cdots,  U_n$  of  $G$  with  $n\in \N$  such that

\vspace{0.7mm}$(1)$\ \ $f(U_i)\subset U_i,$ for each $i\in\N_n$.

\vspace{1mm}$(2)$\ \ Write $U=\bigcup_{i=1}^{\;n}\,U_i$ and $U_0=G-U$. Then $\overline{P(f)}\subset U_0$, $\overline{U}-U\subset EP(f)$,
and $\Omega(f)-U_0=R(f)-U_0=R(f)-\overline{P(f)}\subset U$.

\vspace{1.2mm}$(3)$\ \ For each \,$i\in\N_n$\;\!, $U_i$ has $k_i$ connected components $U_{i1}, \cdots, U_{ik_i}$
with $k_i\in \N$, which satisfy $f(U_{ik_i})\subset U_{i1}$ and $f(U_{ij})\subset U_{i\,,\, j+1}$ for $1\leq j<k_i$.

\vspace{1.2mm}$(4)$\ \ For each \,$i\in\N_n$\;\!, write $W_i=R(f)\cap U_i$. Then $W_i$ is a unique minimal set of $f$
contained in $U_i$.

\vspace{1.2mm}$(5)$\ \ For each \,$i\in\N_n$\;\! and $j\in N_{k_i}$, write $W_{ij}=W_i\cap U_{ij}$. Then $W_{ij}$
is a unique minimal set of $f^{k_i}$ contained in $U_{ij}$, and there is a connected closed subset $G_{ij}$ of
$G$ and a circle $C_{ij}$ such that $W_{ij}\subset C_{ij}\subset G_{ij}\subset U_{ij}$, $f(W_{ik_i})=W_{i1}$,
$f(G_{ik_i})=G_{i1}$, and $f(W_{ij})=W_{i\,,\,j+1}$ and $f(G_{ij})=G_{i\,,\,j+1}$ for $1\leq j<k_i$.

\vspace{1.2mm}$(6)$\ \ For each $i\in\N_n$, $j\in \N_{k_i}$, and for each $x\in U_{ij}$, one has
$\lim_{m\rightarrow\infty}d(f^m(x), W_i)=0$, and $\lim_{m\rightarrow\infty}d(f^{mk_i}(x), W_{ij})=0$.}
\end{thm}

In Section 4, as applications of the main theorem, we give several propositions part of which improve or reprove some known results .
%%%%%%%%%%%%%%%%%%%%%%%%%%%%%%%%%%%%%%%%%%%%%%%%%%%%%%%%%%%%%%%%

%%%%%%%%%%%%%%%%%%%%%%%%%%%%%%%%%%%%%%%%%%%%%%%%%%%%%%%%%%%%%%%%%%
\section{Absorbed sets and orbits of subsets
of graphs}

\begin{defn}\label{absorbed set} Let $f:G\to G$ be a
graph map\;\!. \,For any subset $K$ of\;\! $G$ and any
$n\in\Z_+$\,,\, write  \vspace{2mm}

\noindent(2.1)\hspace{22mm}$O(K,f)=\,\bigcup_{\,i\,=\,0\,}^{\;\infty}f^{\,i}(K)$\,,\
\ \ \ {\rm and}\ \ \
\,$O_n(K,f)=\,\bigcup_{\,i\,=\,0\,}^{\;n}f^{\,i}(K)$\,.
\vspace{2mm}

\noindent Then\;\! $O(K,f)$ is an $f$\;\!-\,invariant set, called
the {\it{orbit of\;\! the set\;\! $K$ under $f$}}, \,and
the set \;\!$O_n(K,f)$\, is called a {\it{segment of\;\!
the orbit of\;\! $K$ under $f$}}\;\!. \,Write  \vspace{2mm}

\noindent(2.2)\hspace{50mm}$O_-(K,f)=\,\bigcup_{\,i\,=\,0\,}^{\;\infty}f^{\,-\;\!i}(K)\,,$
\vspace{2mm}

\noindent called the {\it{inverse orbit of\;\! $K$\! under
$f$}}, or called the {\it{absorbed set by\;\! $K$\! under
$f$}}. \,Let \vspace{2mm}

\noindent(2.3)\hspace{5mm}${\rm Ab}(K,f)=\,\bigcup\;\{\;\!Y\!:\,Y\
{\rm is\ a\ connected\ component\ of}\ O_-(K,f)\;\!,\ {\rm and}\
Y\!\cap \:\!K\!\ne\,\emptyset\,\}$, \vspace{2mm}

\noindent called the {\it{main absorbed set by\;\! $K$\!
under $f$}}. \,

\vspace{2mm}For any connected open subset $U$ of \,$G$,\, define a
function  \vspace{2mm}

\noindent(2.4)\hspace{29mm}$\xi(U)=\xi_G(U)=\sum\big\{_{\;\!}{\rm
val}_{\,G}(v)-2\,:\;v\in U\cap{\rm Br}(G)\big\}$, \vspace{2mm}

\noindent called the {\it{total branching number}}\, of\,
$U$ in\, $G$. \,The following lemma gives the supremum of numbers
of boundary points of connected sets in \,$G$.
\end{defn}

\begin{lem}\label{boundary number}\;(\cite[Lemma 4.1]{MZS})\ \
{\it Let\;\! $G$ be a graph.\, Then\;
$|\partial_GX|\,\le\,\xi(G)+2$ \,for any connected subset\;\!
$X$\;\!\! of \;$G$, \,and \;\!there is a subtree \;$T$\;\!\! of
\;$G$ such that\; $|\partial_GT|\,=\,\xi(G)+2$}\,.
\end{lem}

\begin{cor} \label{boundary number1} {\it Let\;\! $G$ be a
graph, \,and \,$U\!$ be a connected open subset of \;$G$
containing a circle \,$C$. Then\,
$|\partial_GU|\,\le\,\xi(U)$}\;\!.
\end{cor}

\begin{proof} Let\,
$r=\,\min\big\{\;\!d(x,y)/3\,:\,x\;\,{\rm and}\;\,y\;\,\mbox{are
two different points in }\,\partial_GU\cup V(G)\big\}$\:\!, \,and
let
$Z\!=\,U-\,\bigcup\,\{\,\overline{B(x,r)}:\,x\in\,\partial_GU\}$\;\!.
Then $Z$ is a connected open subset of \,$G$ containing \,$C$,
\,$U\!\cap\;\!{\rm Br}(G)=Z\!\cap\;\!{\rm Br}(G)$\;\!, \,and\,
$|\partial_GU|\,\le\,|\partial_GZ|$\;\!. \,Let
$X\!=\overline{Z}$\;\!. Then $X$ is a subgraph of $G$, \,${\rm
Br}(X)=Z\!\cap\;\!{\rm Br}(G)$, \,and
\,$\partial_GZ=\partial_{X\!}Z$\;\!. Take an arc $A\subset C-V(G)$
\,and let $Y\!=Z-A$\;\!. Then $Y$ is a connected open subset of
$Z$\;\!. From Lemma \ref{boundary number} we get \,$|\partial_XZ|+2=|\partial_XY|\le
\xi_{\;\!X\!}(X)+2=\xi_{\;\!G\;\!\!}(Z)+2=\xi_G(U)+2$\;\!. Thus\,
$|\partial_GU|\,\le\,\xi_G(U)$\;\!.
\end{proof}

\begin{lem}\label{absorb boundary} {\it Let $f:G\to G$ be a
graph map\;\!,\, and\, $L\subset G$ be an $f$-invariant connected
set.\, Let\; $W\!=O_-(L,f)$\, and\; $U\!=\,{\rm Ab}(L,f)$ be
defined as in $(2.2)$ and \,$(2.3)$. Then \,

\vspace{0.5mm}$(1)$\ \ Both\; $W$ and\; $U$ are $f$-invariant\;$;$

\vspace{0.5mm}$(2)$\ \ Further,\, if \,$L$ \;\!is open,\, then\;
$W$ and\; $U$ are open, \;$f(\partial_GU)\subset\partial_GU\subset
EP(f)$\;\!, \;and \;$\partial_GU\cap\;\!P(f)$
$\ne\,\emptyset$\;\!}.
\end{lem}

\begin{proof} (1)\ \ Since $f(L)\subset L$ \,and
$f\big(f^{\,-i}(L)\big)\subset f^{\,1-i}(L)$ \,for any
\,$i\in\N$\;\!, \,from the definition of\, $W\!=O_-(L,f)$ we get
$f(W)\subset W\cup f(L)\subset W\cup L=W$. \,Since \,$U$ is the
connected component of\, $W$ containing \,$L$, \,it follows that
$f(U)$ is connected,\, and $f(U)\cap\;\!U\supset f(L)\cap\;\!
L=f(L)\ne\,\emptyset\,$.\, Thus $f(U)\cup U$ is a connected subset
of\, $W$ containing \,$L$ and hence $f(U)\subset U\cup f(U)=U$.
Therefore,\, $W$ and\, $U$ are $f$-invariant.

\vspace{2mm}(2)\ \ Further,\, if \,$L$ is open,\, then
$f^{\,-i}(L)$ is open, \,for any \,$i\in\Z_+$\;\!. \,Thus\, $W$
\;\!is open, so is  the connected component \,$U$ of\, $W$.
From $f(U)\subset U$ we get $f(\:\!\overline
U\:\!)\subset\overline U$.\, If there is a point $x\in\partial_GU$
such that $f(x)\notin\,\partial_GU$, \,then we will have $f(x)\in
U$, and there will be a connected open neighborhood $Z$ of \,$x$\,
such that $f(Z)\subset U$. This means that $Z\subset
f^{-1}(U)\subset f^{-1}(W)\subset W$. Therefore, $Z\cup U$ is a
connected open subset of\, $W$, and hence we have \,$x\in Z\cup
U=\,U$. However, this contradicts that \,$x\in\partial_GU$. Thus we
must have $f(\partial_GU)\subset\partial_GU$, \,which with\,
$|\partial_GU|\,\le\,\xi(G)+2\;$ implies that
\,$\partial_GU\subset EP(f)$ \,and \,$\partial_GU\cap
P(f)\ne\,\emptyset\,$.
\end{proof}

\begin{lem}\label{connect closed} {\it Let $f:G\to G$ be a
graph map\;\!, and \,$K$ be an $f$-invariant connected closed set
with \;$\partial_GK\cap P(f)=\,\emptyset$\;\!.\, Then

\vspace{0.5mm}$(1)$\ \ There exists an $f$-invariant connected
open set $\,L\supset K$ such that the absorbed set\,
$O_-(L,f)=\,O_-(K,f)$\;\!, \,and the main absorbed set\, ${\rm
Ab}(L,f)=\,{\rm Ab}(K,f)\;;$

\vspace{0.5mm}$(2)$\ \ $\,O_-(K,f)$ \,and\; ${\rm Ab}(K,f)$ \,are
open sets in \,$G\;;$

\vspace{0.5mm}$(3)$\ \ For any\, $y\in\,O_-(K,f)$\;\!, \,there
exist an \,$m\in\N$ and a neighborhood\, $Z$ of\, $y$ in \,$G$
such that $f^{\,i}(Z)\subset K$ for all \;$i\ge m\;;$

\vspace{0.5mm}$(4)$\ \
$\omega(f)\cap\;\!O_-(K,f)=\,\omega(f)\cap{\rm
Ab}(K,f)=\,\omega(f)\cap K=\,\omega(f|K)\;;$

\vspace{0.5mm}$(5)$\ \
$\Omega(f)\cap\;\!O_-(K,f)=\,\Omega(f)\cap{\rm
Ab}(K,f)=\,\Omega(f)\cap
K\subset\,\Omega(f|K)\cup\;\!\partial_GK$}\;\!.
\end{lem}

\begin{proof}\ \ (1)\ \ By Lemma \ref{boundary number},
\,$\partial_GK$ is a finite subset of $K$. Write $L_{\:\!0}={\rm
Int}_G(K)$\;\!. Since \,$\partial_GK\cap P(f)=\,\emptyset$\;\!,
\,there is an \;\!$n\in\N$\;\! such that
$f^{\;\!n}(\partial_GK)\subset L_{\:\!0}$\;\!. Let
$L_1=\,\bigcup_{\:\!i=0}^{\:n}f^{\,-i}(L_{\:\!0})$\;\!. Then $L_1$
is an open set containing \,$K$\;\!. \,Let $L$ be the connected
component of\;\! $L_1$ containing \;\!$K$. Then\;\! $L$ is also
open. \,Since $f(L)\cap L\supset f(K)\cap K=f(K)\ne\,\emptyset\,$
and $f(L)\subset f(L_1)\subset
\,\bigcup_{\:\!i=0}^{\;n}\,f^{\;\!1-i}(L_{\:\!0})\subset
K\cup\big(\bigcup_{\;\!i=0}^{\;\!n-1}f^{\,-i}(L_{\:\!0})\big)\subset
L_1$\;\!,\, we have $f(L)\subset f(L)\cup L=L$\;\!.\, From
\;$O_-(L,f)\subset O_-(L_1,f)$
$=\,\bigcup_{j\;\!=0}^{\,\infty}f^{\,-j}\big(\bigcup_{\;\!i=0}^{\:n}f^{\,-i}(L_{\:\!0})\big)
=\,\bigcup_{j\;\!=0}^{\,\infty}f^{\,-j}(L_{\:\!0})
=O_-(L_{\:\!0},f)\subset O_-(K,f)\subset O_-(L,f)$ \,we get
\,$O_-(L,f)=O_-(K,f)$\;\!, \,and hence\, ${\rm Ab}(L,f)=\,{\rm
Ab}(K,f)$\;\!.

\vspace{2mm}$(2)$ \, follows from (1) of this lemma and (2)
of\;\! Lemma \ref{absorb boundary}.

\vspace{2mm}$(3)$\ \ For any\;\! $y\in\,O_-(K,f)$\;\!, \,there is
a \;\!$p\in\N$\, such that $f^{p}(y)\in K$. \,If
$f^{\;\!p}(y)\in\partial_GK$, \,we can take an \;\!$n\in\N$\, such
that $f^{\;\!p+n}(y)\in{\rm Int}_G(K)$\;\! and put
\;\!$m=\,p+n$\;\!. \,If $f^{p}(y)\in{\rm Int}_G(K)$, \,we put
\,$m=\,p$\;\!. \,By the continuity of $f^{\;\!m}$, \,there is a
neighborhood $Z$ of $y$ in $G$ such that $f^{\;\!m}(Z)\subset{\rm
Int}_G(K)$\;\!. \,Since $f(K)\subset K$, we have
$f^{\,i}(Z)\subset K$ for all \,$i\ge m$\;\!.

\vspace{2mm}$(4)$\ \ It suffices to show \,$\omega(f)\cap\;\!
O_-(K,f)\subset\;\!\omega(f|K)$, \,since
\,$\omega(f)\cap\;\!O_-(K,f)\supset\omega(f)$ $\cap\;{\rm
Ab}(K,f)\supset\omega(f)\cap K\supset\omega(f|K)$ \;\!is
clear\;\!. \,Given a point \;\!$x\in\omega(f)\cap
O_-(K,f)$\;\!. \,Since \,$O_-(K,f)$ \,is open,\, there is a point
\,$y\in O_-(K,f)$ \,such that \,$x\in\omega(y,f)$\;\!. Take
\,$n\in\N$ \;such that $f^{\;\!n}(y)\in K$. Then we have
\,$x\in\omega(y_{\;\!},f)=\,\omega\big(f^{\;\!n}(y)\;\!,f\big)\subset\,\omega(f|K)$\,.
Thus \;$\omega(f)\cap\;\! O_-(K,f)\subset\,\omega(f|K)$\;\!.

\vspace{2mm}$(5)$\ \ By (3)\;\! of this lemma, we have
$\big(O_-(K,f)-K\big)\cap\,\Omega(f)=\,\emptyset$\;\!,\, which
implies that\, $\Omega(f)\cap\;\!O_-(K,f)=\Omega(f)\cap{\rm
Ab}(K,f)=\,\Omega(f)\cap K$. \,For any \,$x\in{\rm Int}_G(K)$\;\!,
\,it is clear that \,$x\in\,\Omega(f)$ \,if and only if
\;$x\in\,\Omega(f|K)$\;\!.\, Thus we have \,$\Omega(f)\cap
K\subset\,\Omega(f|K)\cup\partial_GK$.
\end{proof}

\begin{lem}\label{open arc} {\it Let $f:G\to G$ be a
graph map\;\!,\, and\, $K$ be a connected closed subset of\; $G$
with \,$f(K)\cap K\ne\;\emptyset$\,. \,Let\, $X\!=O(K,f)$\, and\,
$X_n\!=O_n(K,f)$ be defined as in $(2.1)$.\, Suppose that\,
$X-X_n\ne\;\emptyset\,$ for any\, $n\in\N$\;\!. Write\,
$S=\,\bigcap_{\;\!n=0}^{\,\infty}\overline{X-X}_{\!n\,}$.\, Then

\vspace{0.5mm}$(1)$\ \ \,$P(f)\;\!\supset
f(S)\;\!=\,S\;\!\ne\;\emptyset$\,, \,and $\,S$ contains at most
$\:\xi(G)+2$ \,points\;$;$

\vspace{0.5mm}$(2)$\ \ Further, \,if\, $K\cap
EP(f)=\;\emptyset$\;\!, \,then there exists an open arc
\,$(u,v)\subset X-V(G)$ \,such that \;$S=\overline
X-X=\{v\}\subset {\rm Fix}(f)$\;\!,\, $f(x)\in(x,v)$ for any
\,$x\in[u,v)$\;\!, \;and \;$X$ is contained in the main absorbed
\;\!set\, ${\rm Ab}\big((u,v)\;\!,f\;\!\big)$}.
\end{lem}

\begin{proof}\ \ (1)\ \ By the conditions of the
lemma,\, all the sets $K\!=X_0\subset X_1\subset X_2\subset
X_{3}\subset\cdots\subset X$ are connected,\, and \;\!$X_n$ \;\!is
closed in \;\!$G$ \;\!for all \,$n\in\Z_+$\;\!. \;\!Since\,
$X-X_n\ne\;\emptyset\,$, \,for any \,$n\in\N$\,,\, we have\,
$X_{n+1}-X_{n}\ne\,\emptyset$\,. Since $V(G)$ is a finite set and
since the valence of any vertex of\:\! $G$ is finite, there is a
sufficiently large\, $\beta\in\N$ \,such that

\vspace{2mm}{{(a)}} \ $X\cap V(G)=X_n\cap V(G)=X_\beta\cap V(G)$
\,for any \,$n\ge\beta$\,;

\vspace{1mm}{{(b)}} \ ${\rm val}_{X_n}(w)=\,{\rm
val}_{X_\beta}(w)$ \,for any \,$w\in X_\beta\cap V(G)$ \,and any
\,$n\ge\beta$\;\!.

\vspace{2mm}\noindent From property (a) we can derive

\vspace{2mm}{{(c)}} \ For any \,$n\ge\beta$\,, \,every connected
component of $X-X_{n}$\;\! is contained in an edge of \;\!$G$,
\,and every edge of \;\!$G$ \;\!contains at most two connected
components of $X-X_{n}$\;\!. Thus $X-X_{n}$\;\! has only finitely
many connected components\,;

\vspace{1.5mm}{{(d)}} \ For any \,$n\ge\beta$\,, \,every connected
component of $X-X_{n+1}$\;\! is contained in a connected component
of $X-X_{n}$\;\!, \,and every connected component of $X-X_{n}$\;\!
contains at most one connected components of $X-X_{n+1}$\;\!.

\vspace{2mm}Let $\lambda_{\,n}$ be the number of connected
component of $X-X_{n}$\;\!. By property (d) we have
$\lambda_{\beta}\,\ge\,\lambda_{\beta+1}\,\ge\,\lambda_{\beta+2}\,\ge\,\cdots$.\;
Let $\lambda=\lim_{\,n\to\infty}\lambda_{\,n}$\;\!. Then
$\lambda\ge1$, and there is an\, $m\ge\beta$\;\! such that
$\lambda_{\,n}=\lambda$ \,for all\;\! $n\ge m$. \,If\,
$m>\beta$\;\!,\, we can replace $\beta$\, by $m$\;\!. Thus we may
assume that $\lambda_{\,n}=\lambda$ \,for all\, $n\ge\beta$\;\!.
This means that, \,for\, $n\ge\beta$\;\!,\, every connected
component of $X-X_{n}$\;\! contains exactly one connected
components of $X-X_{n+1}$\;\!.

\vspace{2mm}Since $X$ is connected and $X_\beta$ is closed, \,no
connected component of $X-X_{\beta}$\;\! is a closed arc. \,If
there is a connected component $J$ of \,$X-X_{\beta}$\, such that
$J=(u,v]$\, is a semi-open arc, \,then \,$u\in X_\beta$ \,and
there is a neighborhood\;\! $U$ of \;\!$v$\;\! in $G$ such that
\,$X\!\cap \;\!U\subset(u,v]$\;\!. \,Since\, $v\in X$,\, there
is\, $m>\beta$\, such that\, $v\in X_m$\,, \,which with
\,$X_m\cap\;\! U\subset X\cap\;\!U\subset(u,v]$ \,and \,$(u,v)\cap
V(G)=\,\emptyset$ \,implies that\, $[u,v]\subset X_m$\;\!, \,and
hence $J=(u,v]$\, contains no connected component of
\,$X-X_{m}$\;\!. \,This  leads to a contradiction. Thus
every connected component of\;\! $X-X_{\beta}$\;\! must be an open
arc.

\vspace{2mm}Let \,$J_1\;\!,\,\cdots,\,J_\lambda$\, be the
\,$\lambda$\, connected components of $X-X_{\beta}$\, with
\,$J_i=(u_i,v_i)$\;\!. For each \,$i\in\N_\lambda$\;\!,\, by means
of a homeomorphism from \,$[u_i,v_i]$ \,to \,$[\:\!0,1]$ \,we can
define a linear order \,$<$ \,on \,$[u_i,v_i]$ \,such that
\,$u_i<v_i$\;\!. For any \;\!$n\ge\beta$\;\!,\; write\,
$J_{in}=J_i-X_n$\;\!. Then $J_{in}$ \;\!is an open arc,\, and
$J_{in}$ \;\!is the connected component of \;\!$X-X_{n}$\;\!
contained in \,$J_i$\;\!. \,Suppose that
$J_{in}=(u_{in}\;\!,v_{in})$ \,with \,$u_{in}<v_{in}$\;\!. Then we
have \;$\lim_{\,n\to\infty}d(u_{in},v_{in})=0$ \;and \vspace{-2mm}
$$
u_{i}=u_{i\beta}\le u_{in}\le u_{im}<v_{im}\le v_{in}\le
v_{i\beta}=v_{i}\ \ \ \ {\rm for\,\ any}\ \ \  m>n\ge\beta\;\!.
\vspace{-2mm}
$$
Let \,$z_i=\lim_{\,n\to\infty}u_{in}$\;\!. \,If
\,$u_{in}<z_i<v_{in}$ \,for all \,$n\ge\beta$\;\!, \,then
\,$z_i\notin X_n$ \,for all \,$n\ge\beta$\;\!. This
contradicts that \,$z_i\in J_i\subset X$. \,Thus there exists
\,$m\ge\beta$ \,such that \,$u_{im}<z_i=v_{im}$ \,or
\,$u_{im}=z_i<v_{im}$\;\!. By symmetry, \,we may assume that the
case \,$u_{im}<z_i=v_{im}$ \,occurs. \,In addition, if $m>\beta$
then we can replace $\beta$\, by $m$\;\!. \,Hence we may assume
that,

\vspace{2mm}{{(e)}} \ for any \;\!$i\in\N_\lambda$\, and any
\;\!$n\ge\beta$\;\!, \;$J_{in}=(u_{in},v_{i})$\;\!, \;and
\;$\lim_{\,n\to\infty}u_{in}=v_{i}$\;\!.

\vspace{2mm}From property (e) we get
$X-X_n=\,\bigcup_{i=1}^{\,\lambda}(u_{in\;\!},v_i)$\;\!,\;
$S=\,\bigcap_{\;\!n=0}^{\,\infty}\overline{X-X}_{\!n}=\{v_1\;\!,\,\cdots,\,v_\lambda\}$\;\!,
\,and\, $\overline{X}=X\cup S$. \,By property (b)\;\!,\, we have
\,$u_i\notin V(G)$\;\!.\, Note that it is possible that\, $v_i\in
V(G)\cup X_\beta$ \,for some $i\in\N_\lambda$\;\!,\; or that\,
$v_i=v_j$\, for some\, $1\le i<j\le\lambda$\;\!.\, Let \,$E_i$\,
be the edge of \,$G$ containing\;\! $(u_i,v_i)$ \;\!and let\;\!
$w_i$ \;\!be the endpoint of \;\!$E_i$\;\! such that\;\!
$u_i\in(w_i,v_i)$\;\!. \,Let
$\,\varepsilon=\;\!\min\;\!\{d(u_i,w_i):\,i\in\N_\lambda\}$\;\!.
Take a\, $\delta>0$ \,such that\,
$d\big(f(x),f(y)\big)<\;\!\min\;\!\{\varepsilon,1\}$ \;\!for
any\;\! $x\;\!,\;\!y\in G$ \,with\, $d(x,y)\le\delta$\;\!. \,We
may assume that \;\!$\beta$\, is so large that

\vspace{2mm}{{(f)}} \ all the diameters of
\,$J_1\;\!,\,\cdots,\,J_\lambda$\, are less than $\,\delta$.

\vspace{3mm}For any $i\in\N_\lambda$, \,it is clear that
\vspace{1.5mm}

\hspace{3mm}$J_{i,\,\beta+1}=J_i-X_{\beta+1}\,\subset\,X-X_{\beta+1}\,\subset
\,f(X-X_{\beta})\,=\,f\big(\bigcup_{k=1}^{\,\lambda}J_k\big)\,=\;\bigcup_{k=1}^{\,\lambda}f(J_k)\,.
\vspace{1.5mm}$

\noindent So for each\;\! $i\in\N_\lambda$\;\! there is a
\;\!$k_{\;\!i}\in\N_\lambda$\, such that $f(J_{k_i})\cap
J_{i,\,\beta+1}\ne\,\emptyset$\,, which with property (f) implies
that\, $w_i\notin f(J_{k_i})$\,. \,If\;\! $v_i\in f(J_{k_i})$\;\!,
\,then there exist\;\! $x_i\in J_i$ \;\!and\;\! $m>\beta$ \;\!such
that $\{x_i,v_i\}\subset f(X_m\cap J_{k_i})$\,, which leads to\,
$[x_i,v_i]\subset f(X_m\cap J_{k_i})\subset f(X_m)\subset
X_{m+1}$. But this contradicts that
$J_i-X_{m+1}=(u_{i,m+1},v_i)$\,. So we must have\, $v_i\notin
f(J_{k_i})$\,, which with $w_i\notin f(J_{k_i})$ implies that
$f(J_{k_i})\subset(w_i,v_i)$\,. Hence, \,for any
$i,\,j\in\N_\lambda$\;\! with \,$j\ne i$\;\!, \,we have
$f(J_{k_j})\cap f(J_{k_i})\subset(w_{j\,},v_j)\cap
(w_{i\;\!},v_i)\,=\;\emptyset\,$,\, which implies that\, $k_j\ne
k_{\;\!i}$\,.\, Thus we have

\vspace{3mm}{{(g)}} \
$\{k_{\;\!i}:i\in\N_\lambda\}=\N_\lambda$\;\!, \,and\;\!
$J_{i,\,\beta+1}\;\!\subset f(J_{k_i})\;\!\subset(w_i,v_i)$\, for
each\, $i\in\N_\lambda$\;\!.

\vspace{3mm}For $n\ge\beta$ and $i\in\N_\lambda$,\, since
$J_{i,\,n+1}\subset X-X_{n+1}\subset
f(X-X_{n})=\,\bigcup_{k=1}^{\,\lambda}f(J_{kn})$\;\!,\, from
property (g) \;\!we get\;\! $J_{i,\,n+1}\subset f(J_{k_in})$\;\!,
\,which with  \vspace{-3mm}
$$
\mbox{$\lim_{\,n\to\infty}{\rm
diam}\big(f(J_{k_in})\big)=\,\lim_{\,n\to\infty}{\rm
diam}(J_{k_in})=\,0$} \vspace{-2.5mm}
$$
implies that $f(v_{k_i})=v_i$\;\!. \,Hence we have $f(S)=S$.
\,Noting that $S=\{v_1\;\!,\,\cdots,\,v_\lambda\}$ is a finite
set, \,we have $S\subset P(f)$\;\!. \,By Lemma \ref{boundary number}, \,we have
\vspace{-2mm}
$$
|S|\le\lambda=|\{u_1\:\!,\,\cdots,\,u_\lambda\}|=|\partial_XX_\beta|\le|\partial_GX_\beta|\le\xi(G)+2\;\!.
$$

\vspace{1mm}(2)\ \ \,Further,\, if
\,$K\cap\;\!EP(f)=\,\emptyset$\,, \,then \,$X\cap S\subset X\cap
P(f)\subset X\cap EP(f)=\,\emptyset$\,, \,which with \;$\overline
X=X\cup S$ \,implies that \;$\overline X-X=\,S$\;\!.

\vspace{2mm}For any $i\in\N_\lambda$ \,and any $n\ge\beta$,
\,write\, $L_{in}=L(i,n)=J_i\cap(X_{n+1}-X_n)$\,.\vspace{0.5mm}
Then\, $L_{in}=(u_{in}\:\!,u_{i,n+1\,}]\subset J_{in}$ \,and\,
$X_{n+1}-X_{n}=\,\bigcup_{i=1}^{\,\lambda}L_{in\,}$.\vspace{0.5mm}
\,Noting that\, $L_{i,n+1}\subset X_{n+2}-X_{n+1}$ $\subset
f(X_{n+1}-X_{n})=\bigcup_{k=1}^{\,\lambda}f(L_{kn})$\;\!, \,from
property (g) we get\, $L_{i,n+1}\subset f(L_{k_in})$\;\!.

\vspace{2mm}Define a map $\psi:\N_\lambda\to \N_\lambda$\;\! by
$\psi(i)=k_{\;\!i}$\;\! for any $i\in\N_{\lambda\,}$. Then $\psi$
is a bijection and $f(v_{\:\!\psi(i)})=v_i$\;\!.\, Let\,
$t_i\in\N_\lambda$ \,be the least positive integer such that\,
$\psi^{t_i}(i)=i$\;\!. Then we have $f^{\,t_i}(v_i)=v_i$\;\!.\,
Choose an\, $n>\beta+t_i$ \,such that \,$L_{in}\ne\,\emptyset\,$
and take a point \,$z_0\in L_{in\,}$. Then there exist points
\,$z_1\,,\,z_2\,,\,\cdots,\,z_{t_i}$\, such that
\,$f(z_j)=z_{j-1}$ \,and \,$z_j\in
L\big(\psi^{\,j}(i)\;\!,\,n-j\:\!\big)$\, for each
\,$j\in\N_{t_i\,}$. \,Noting that \,$z_{t_i}\in
L(i,\:\!n-t_i)\subset (u_i\:\!,u_{in\,}]$ \,and
\,$z_0\in(u_{in}\:\!,u_{i,n+1\,}]$\,, \,we have
\,$f^{\,t_i}(z_{t_i})=z_0\in(z_{t_i\,},v_i)\subset(u_i,v_i)$\:\!.

\vspace{3mm}If
$f^{\,t_i}\big((u_i,v_i)\big)\not\subset(u_i,v_i)\,$, then there
exist \:\!$m>n$ \,and\:\! $z\in(u_i,u_{im}]=(u_i,v_i)\cap X_m$
\:\!such that $f^{\,t_i}(z)\in\{u_{i\;\!},v_i\}$ and
$f^{\,t_i}\big([z_{t_i},z)\big)\subset(u_i,v_i)\;\!$. However,\,
if $f^{\,t_i}(z)=u_i$\, then $(z_{t_i},z)\cap {\rm
Fix}(f^{\,t_i})\ne\,\emptyset$\,, \,which contradicts that
$[z_{t_i},z]\cap P(f)\subset X\cap EP(f)=\,\emptyset\,$. \,If
$f^{\,t_i}(z)=v_i$\, then $X_{m+t_i}\supset
f^{\,t_i}\big([z_{t_i},z]\big)\supset[z_0,v_i]$\,, which also
contradicts that\:\!
$J_i-X_{m+t_i}=(u_{i,\:\!m+t_i\,},\;\!v_i)$\,. Thus we must have
$f^{\,t_i}\big((u_i,v_i)\big)\subset(u_i,v_i)\,$, which with
$f^{\,t_i}(z_{t_i})\in(z_{t_i},v_i)$\;\! and $[u_i,v_i)\cap
P(f)=\,\emptyset\,$ implies that $f^{\;\!t_i}(x)\in(x,v_i)$ \,for
any $x\in[u_i,v_i)$\;\!.

\vspace{2mm}Write \,$t=\,\prod_{j\,=1}^{\;\lambda}t_j$\;\!. Then
for any $i\in\N_{\lambda}$\;\!, \,the open arc $J_i=(u_i,v_i)$ is
$f^{\,t}$-invariant. \,Let\, $U_i=\:\!{\rm Ab}(J_i,f^{\,t})$ \,be
the main absorbed set by $J_i$\;\! under $f^{\,t}$\;\!. Then by
Lemma \ref{absorb boundary} we get \,$\partial_GU_i\subset EP(f^{\,t})=EP(f)$\,.
\,Hence, \,if $X\not\subset U_i$ \,then \,$X\cap EP(f)\supset
X\cap\,\partial_GU_i\ne\,\emptyset$\;\!. \,But this will leads to
a contradiction. Thus we must have \,$X\subset U_i$\;\!.

\vspace{2mm}If \,$\lambda\ge2$\;\!, \,then we have both
\,$X\subset U_1$ \,and \,$X\subset U_2$\,. \,On the other hand,
\,from \,$J_1\cap J_2=\,\emptyset$ \,we get \,$U_1\cap
U_2=\,\emptyset$\;\!. \,These will lead to a contradiction. Thus
we must have \,$\lambda=1$\;\!, \,which implies that \,$t=t_1=1$.
\,Let \,$u=u_1$ \,and \,$v=v_1$\;\!. \,Then the open arc \,$(u,v)$
\,satisfies the all conditions in Lemma \ref{open arc}, \,and the proof is
complete.
\end{proof}

As a corollary of Lemma \ref{open arc}, \,we have the following

\begin{prop}\label{empty period} \ {\it Let $f:G\to G$ be
a graph map\;\!,\, and\, $K$ be a connected closed subset of\; $G$
\;\!with $f(K)\cap K\ne\;\emptyset$\,. \,Let\, $X\!=\,O(K,f)$\, be
the orbit of \,$K$ under $f$. \,If\; $\overline X-X$ contains more
than one point, \,then \,$X\cap P(f)\ne\,\emptyset$}\;\!.
\end{prop}

\begin{proof} Let $X_n=\,O_n(K,f)$ be defined
as in $(2.1)$\;\!. Then $X_n$ is closed. Since $X$ is not closed,
\,we have $X-X_n\ne\;\emptyset\,$ for any $n\in\N$\;\!. \,If
\;\!$X\cap P(f)=\,\emptyset$\;\!, \,then $X\cap
EP(f)=\,\emptyset$\;\!, \,and by Lemma \ref{open arc},\; $\overline X-X$ will
contain only one point.  \,Therefore, \,if\; $\overline X-X$
contains more than one point then we must have $X\cap
P(f)\ne\,\emptyset$\;\!.
\end{proof}

\begin{lem}\label{set orbit} {\it Let $f:G\to G$ be a
graph map\;\!,\, and\, $K$ be a connected closed subset of\, $G$
with \,$f(K)\cap K\ne\;\emptyset$\,. \,Let\, $X\!=\,O(K,f)$\,
and\, $X_n=\,O_n(K,f)$ \;\!be defined as in $(2.1)$. \,If\,
$X-X_n\ne\;\emptyset\,$ for any\, $n\in\N$\;\!,\; then the
following eight conditions are equivalent\,$:$

 $(1)$\ \ $K\cap EP(f)=\,\emptyset\,;$

 $(2)$\ \ $X\cap EP(f)=\,\emptyset\,;$

 $(3)$\ \ $X\cap P(f)=\,\emptyset\,;$

 $(4)$\ \ $X\cap R(f)=\,\emptyset\,;$

 $(5)$\ \ $X\cap\,\omega(f)=\,\emptyset\,;$

 $(6)$\ \ $X\cap\,\Omega(f)=\,\emptyset\,;$

 $(7)$\ \ $\overline X-X\ne\,\emptyset\,$, \,and
\;$\lim_{\;n\to\infty\,}d\big(f^{\;\!n}(x)\,,\;\!\overline
X-X\big)=\,0$ \,for any \;$x\in X\,;$

 $(8)$\ \ There exists an open arc \,$(u,v)\subset X-V(G)$ \,such
that \;$\overline X-X=\{v\}\subset {\rm Fix}(f)$\;\!,\,
$f(x)\in(x,v)$ for any \,$x\in[u,v)$\;\!, \;and
\;$X\subset\;\!{\rm Ab}\big((u,v)\;\!,f\;\!\big)$}\;\!.
\end{lem}

\begin{proof} (6) $\Rightarrow$ (5) $\Rightarrow$
(4) $\Rightarrow$ (3)\, are trivial, \,since\, $\Omega(f)\supset
\omega(f)\supset R(f)\supset P(f)$\;\!.\, By the definition of the
orbit $X\!=\,O(K,f)$\;\!, \,(3) $\Leftrightarrow$ (2)
$\Leftrightarrow$ (1)\;\! are clear.\, Since $\overline
X-X\subset\partial_GX$ is a finite set,\, from the definition of
recurrent points we can directly derive \,(7) $\Rightarrow$ (4).
\,In Lemma 2.6 we have proved (1) $\Rightarrow$ (8). \,Thus it
suffices to show that (8) $\Rightarrow$ (6) \,and\, (8)
$\Rightarrow$ (7)\;\!.

\vspace{1mm}Suppose that (8) is true. Then
\,$(u,v)\cap\,\Omega(f)=\,\emptyset$\;\!, \,and
\;$\lim_{\;n\to\infty\,}f^n(x)=v$ \;for any\, $x\in[u,v)$\;\!.\,
Since \,$X\!\subset\;\!{\rm Ab}\big((u,v)\;\!,f\;\!\big)$\;\!,\,
we also have \,$X\cap\,\Omega(f)=\,\emptyset$\;\!,\, and
\;$\lim_{\;n\to\infty\,}f^n(x)=v$ \;for any\, $x\in X$\;\!. \,Hence
(6) and (7) are true.
\end{proof}

\vspace{4mm}From Lemma \ref{set orbit} we obtain the following corollary at
once.

\begin{cor}\label{orbit closed} {\it Let $f:G\to G$ be a
graph map\;\!,\, and\, $K$ be a connected closed subset of\, $G$
with \,$f(K)\cap K\ne\;\emptyset$\,. \,Let\, $X\!=O(K,f)$\, and\,
$X_n=\,O_n(K,f)$\;\! be defined as in $(2.1)$\;\!. \,If\, $K\cap
EP(f)=\,\emptyset\;\!$ and\: $K\cap\, \Omega(f)\ne\,\emptyset\,$,
\,then there exists an\, $n\in\N$\;\! such that\;\!
$X\!=X_{n\,},$\; and\;\! hence\, $X$ is closed in\, $G$.}
\end{cor}

\vspace{3mm}The following lemma is given in \cite{MS2}.

\begin{lem}\label{arc intersection} $($\cite[Lemma 2.3]{MS2}$)$\ \
{\it Let $f:G\to G$ be a graph map\;\!, and $A=[w,z]$ be an arc
in\, $G$. If\, $(w,z)\cap\,\big(V(G)\cup P(f)\big)=\,\emptyset$\,
and \,$(w,z)\cap\,R(f)\ne\,\emptyset$, then $O(w,f)\cap
(w,z)\ne\,\emptyset$. }
\end{lem}

\begin{cor}\label{empty eventual} {\it Let $f:G\to G$ be a
graph map\;\!, \,and $A=[x,y]$ be an arc in $G$. \,If\,
$(x,y)\,\cap\,\big(V(G)\cup P(f)\big)=\,\emptyset\,$ and
\;$(x,y)\,\cap\,R(f)\ne\;\emptyset$\,, \,then\, $A\cap
EP(f)=\,\emptyset\,.$ }
\end{cor}

\begin{proof} If $A\cap\;\!EP(f)\ne\,\emptyset$,
then there exist $w\in A$ and $n\in\N$ such that $f^{\;\!n}(w)\in
P(f)$\;\!, which with $(x,y)\cap P(f)=\,\emptyset\,$ implies that
$f^{\,i}(w)\notin(x,y)$ \;\!for any $i\ge n$\;\!. On the other
hand, take a point $z\in[x,y]-\{w\}$ such that $(w,z)\cap
R(f)\ne\,\emptyset$\;\!. Then $(w,z)\cap
R(f^{\;\!n})\ne\,\emptyset$\;\!, \,and by Lemma \ref{arc intersection} we get
$O(w,f^{\;\!n})\cap (w,z)\ne\,\emptyset$\;\!. But this contradicts
that $f^{\,i}(w)\notin(x,y)$ \;\!for any $i\ge n$\;\!. Thus we
must have\, $A\cap EP(f)=\,\emptyset$\;\!. \end{proof}

%%%%%%%%%%%%%%%%%%%%%%%%%%%%%%%%%%%%%%%%%%%%%%%%%%%%%%%%%%%%%%%%%%
\section{Structures of $R(f)-\overline{P(f)}$
for graph maps $f$}

Now we list several results coming from \cite{MS1} in the following theorem,
 which is a key ingredient in the proof of the main theorem.

\begin{thm}\label{eventual invariant} {\it Let $f:G\to G$ be a
graph map\;\! without a periodic point. Then there exist an
\,$n\in\N$\,, \,a subgraph $X$ \!of \;$G$ and a circle \,$Q\subset
X$ such that the following items hold:

$(1)$  \,$G\,\supsetneq\,f(G)\,\supsetneq\,\cdots\,\supsetneq
\,f^{\,n-2}(G)\,\supsetneq
\,f^{\,n-1}(G)\,=X=f(X)$\;\!;

$(2)$  $f|X$ has a unique minimal set $M$ which is totally minimal;

$(3)$  $\Omega(f|X)=AP(f|X)=M\subset Q$};

$(4)$ for each $x\in X$, we have $\lim_{m\rightarrow\infty}d(f^{m}(x), M)=0$;

$(5)$ $f$ is topologically semi-conjugate to an
irrational rotation of \;\!the unit circle $S^1$.
\end{thm}

\begin{rem}
The conclusion $(1)$ of Theorem \ref{eventual invariant} follows from Theorem \cite[Theorem 4.2]{MS1}
and Definitions 3.3 and 3.4 in \cite{MS1}; the conclusions (2) and (3) follow from Claim 24 and Corollary 4.4 in \cite{MS1};
the conclusion $(5)$ is implied by \cite[Theorem 4.3]{MS1}.
Though the conclusion $(4)$ of Theorem \ref{eventual invariant} is not explicitly stated
in \cite{MS1}, it can be seen easily from \cite[Theorem 4.2]{MS1} and the constructions in Section $3$ of \cite{MS1}.
\end{rem}

\begin{defn}\label{component cyclic} Let $f:G\to G$ be a
graph map\;\!. A subset $X$ of $G$ is called a
{\it{component-cyclic\! $f$-invariant set}} \;if $X$ has
only finitely many connected components $X_1\,,\,\cdots,\,X_k$
\;and $f(X_{i})\subset X_{i+1\,({\rm mod}\;k)}$ \,for every\,
$i\in\N_k$\;\!.\, A component-cyclic $f$-invariant set $X$ is
called a {\it{component-cyclic strongly $f$-invariant set}}
\,if $f(X)=X$.
\end{defn}

\begin{lem}\label{union component} {\it Let $f:G\to G$ be a
graph map\;\!, \,and \,$X$,\,$Y$\! be component-cyclic
$f$-invariant sets. \,If\:\! $X\!\cap Y\!\ne\emptyset\;\!$, \,then
\:\!$X\!\cup Y\!$ is also a \;\!component-cyclic $f$-invariant
set}\;\!.
\end{lem}

\begin{proof} Suppose that $X$ has \,$n$
\,connected components $X_1\,,\,\cdots,\,X_n$\;\! with
$f(X_{i})\subset X_{i+1\,({\rm mod}\;n)}$ \,for\, $i\in\N_{n\,}$,
\;and\, $Y$ has \,$m$\;\! connected components
\,$Y_1\,,\,\cdots,\,Y_m$\;\! with $f(Y_{i})\subset Y_{i+1\,({\rm
mod}\;m)}$ \,for\, $i\in\N_m$\;\!.\, Since
$X\cap\,Y\ne\,\emptyset\;\!$,\, we may assume that
$X_1\cap\,Y_1\ne\,\emptyset\;\!$.\, For any $i\in\N$\;\!,\, write
$X_{i+n}=X_i$\, and\, $Y_{i+m}=Y_{i\,}$. Then
$X_i\cap\,Y_i\ne\,\emptyset$\;\!. \,Let $Z=X\cup\,Y$ and let $Z_i$
be the connected component of $Z$ containing $X_i\cup\,Y_i$\;\!.
Then\, $f(Z_{i})\subset Z_{i+1\,}$, \,and there is a common
factor\;\! $k$\, of\;\! $n$\;\! and\;\! $m$\;\! such that
$Z_{i+k}=Z_{i\,}$ for all\;\! $i\in\N$\;\!. Thus $Z=X\cup\,Y$ is
also a component-cyclic $f$-invariant set.
\end{proof}

\begin{lem}\label{component exist} {\it Let $f:G\to G$ be a
graph map\;\!. \,Then for any \,$x\in R(f)-\overline{P(f)}$
\,there is a component-cyclic $f$-invariant closed set \,$Y$\!
such that \,$x\in Y\!\subset G-EP(f)$}\;\!.
\end{lem}

\begin{proof} Since \,$x\in
R(f)-\overline{P(f)}$\;\!, there is an arc \,$A=\,[x,y]\subset
G-P(f)$\, such that\, $(x,y]\cap V(G)=\,\emptyset\;$ and\,
$(x,y)\cap R(f)\supset(x,y)\cap\;\!O(x,f)\ne\,\emptyset\,$.
\,Hence, there is an\;\! $n\in\N$ such that $f^{\;\!n}(A)\cap
A\ne\,\emptyset\;\!$. \,Let $Y\!=\,O(A,f)$\, and\,
$Z=\,O(A,f^{\;\!n})$ \,be the orbits of $A$ under $f$ and
$f^{\;\!n}$, respectively. Then $f(Y)\subset
Y\!=\,\bigcup_{\:i=0}^{\,n-1}f^{\,i}(Z)$\;\!,
\,$f^{\;\!n}(Z)\subset Z$, \;and $Z$\;\! is connected. \,By
Corollary \ref{empty eventual}, \,we have $A\cap EP(f)=\,\emptyset$\;\!. Hence
\;$Y\!\subset G-EP(f)$\;\!. \,By Corollary \ref{orbit closed}, \,$Z\;\!$ is
closed in \;\!$G$.\, So\;\! $Y\!$ \,is also closed. \,Let $k$\;\!
be the number of connected components of \;\!$Y$\!. Then $k\;\!$
is a factor of $\;\!n$\;\!. \,Let \;\!$Y_1$\;\! be the connected
component of\, $Y$ containing \,$Z$\;\!. Then\, $x\in Y_1$ \,and\,
$f^{\;\!k}(Y_1)\subset Y_1$\,. \,Write\, $Y_i=f^{\,i-1}(Y_1)$
\,for \,$i=2\,,\,\cdots,\,k$\;\!. \,Then\,
$Y_1\,,\,Y_2\,,\,\cdots,\,Y_k$ \,are just the \,$k$ connected
components of\, $Y$\!, \,and $f(Y_{i})\subset Y_{i+1\,({\rm
mod}\;k)}$ \,for every\, $i\in\N_k$\;\!. Thus \,$Y$ is a
component-cyclic $f$-invariant set.
\end{proof}

\vspace{3mm}For any graph map $f:G\to G$ \,and any\,
$n\in\N$\;\!,\, write \vspace{-2mm}
$$
EP_n(f)=\;\!\{x\in G:\ {\rm the\ orbit}\;\,O(x,f)\;\,{\rm
contains\ at\ most}\;\,n\;\,{\rm points}\}.    \vspace{-2mm}
$$
Then \;\!$EP_n(f)$\;\! is a closed subset of \,$G$,\, and
\;\!$EP_n(f)\subset EP(f)$\;\!. For any connected open set \,$U$
in\, $G$, \,let\, $\xi(U)$\, be defined as in (2.4)\;\!. The
following proposition describes the structures of component-cyclic
$f$-invariant closed sets without periodic points.

\begin{prop}\label{sub main} {\it Let $f:G\to G$ be a
graph map\;\!, \,and \;$Y\!\subset G-P(f)$ \;\!be a
component-cyclic\;\!\! $f$-invariant closed set with \,$k$
connected components\, $Y_1\,,\,\cdots,\,Y_k$ \,such that\:\!
$f(Y_{i})\subset Y_{i+1\,({\rm mod}\;k)}$ \,for each\,
$i\in\N_k$\;\!.\, Then

\vspace{1mm}$(1)$\ \ There exists a component-cyclic \,strongly\,
$f$-invariant closed set \,$X$ with \,$k$\;\! connected
components\, $X_1\,,\,\cdots,\,X_k$ \,such that \,$X_{i}\subset
Y_{i}$ \,for each\; $i\in\N_k$\;$;$

\vspace{1mm}$(2)$\ \ For any \,$i\in\N_k$\;\!, \,let \;$U_i={\rm
Ab}(X_{i\,},f^{\,k})$ \,be the main absorbed set by\;\! $X_i$
under $f^{\,k}$\!, \,and let\;
$U\!=\,\bigcup_{\;\!i=1}^{\,k}U_i$\;\!. Then \,$Y_i\subset
U_i$\;\!, \,$f(U_i)\subset U_{i+1\,({\rm mod}\;k)}$\;\!,
\,$Y\!\subset U\!={\rm{Ab}}(X,f)$\;\!, \,and \;$U$ is a
component-cyclic $f$-invariant open set with \,$k$\;\! connected
components\, $U_1\,,\,\cdots,\,U_k$\;$;$

\vspace{1mm}$(3)$\ \ For any \,$i\in\N_k$\;\!,\, $X_i$\;\!
contains at\;\! least one circle\, $C_i$\;\!,\, and
$f^{\,k}|X_i$\;\! has a unique minimal set \;\!$M_i$\;\!,\, which
satisfy\; $\omega(f)\cap\;\!U_i=AP(f^{\,k}|X_i)=\,M_i\subset C_i$
\,and\; $\Omega(f)\cap\;\!U_i=\,\Omega(f^{\,k})\cap X_i\subset
M_i\cup\;\!\partial_GX_i$\;\!. This $M_i$ is also a unique minimal
set of\;\! $f^{\,k}|Y_i$ \;\!and of\;\! $f^{\,k}|U_i\;;$

\vspace{1mm}$(4)$\ \ Write\,
$M=\,\bigcup_{\;\!i=1}^{\,k}\;\!M_i$\;\!. Then \
$\omega(f)\cap\;\!U=AP(f|X)=AP(f^{\,k}|X)=M$,\;
$\Omega(f)\cap\;\!U\!=\,\Omega(f)\cap X=\,\Omega(f^{\,k})\cap
X\subset\,M\cup\;\!\partial_GX$\;\!, \,and\, $M$ is\;\! a unique
minimal \;\!set of any of \;\!the three maps \;\!$f|X$,\;\! $f|Y$
and \;\!$f|\:\!U\;;$

\vspace{1mm}$(5)$\ \ For any \,$i\in\N_k$\;\!, \,the main absorbed
set\, $U_i={\rm{Ab}}(X_{i\,},f^{\,k})$\;\! is a connected
component of any of \;\!the three sets\; $G-EP(f)$\;\!,\;
$G-\overline{EP(f)}$ \,and\; $G-EP_{\xi(G)}(f)$}\;\!.
\end{prop}

\begin{proof} Since \,$Y$ is $f$-invariant,\,
from \,$Y\!\subset G-P(f)$ \,we get \,$Y\!\subset G-EP(f)$\;\!.

\vspace{1mm}(1)\ \ Note that\;\! $Y_1$\;\! itself is a graph,
\,and $f^{\,k}|Y_1:Y_1\to Y_1$\;\! is a graph map without periodic
point. By $(1)$ of Theorem \ref{eventual invariant}, \,there exist an \,$n\in\N$ \,and a
connected closed set $X_1$ $\subset Y_1$\;\! such that
$f^{\,k(n-1)}(Y_1)=X_1=f^{\,k}(X_1)$\;\!.\, Let\,
$X_i=f^{\,i-1}(X_1)$ \,for \,$i=2\,,\,\cdots,\,k$\;\!.\, Put\;
$X=\,\bigcup_{\;\!i=1}^{\,k}\;\!X_i$\;\!. Then $X$ satisfies the
conditions mentioned in (1) of this theorem.

\vspace{2mm}(2)\ \ \,For any \,$i\in\N_k$\;\!, \,it follows from
(2) of\;\! Lemma \ref{connect closed} that \,$U_i={\rm{Ab}}(X_{i\,},f^{\,k})$ \,is
open in $G$. \,Since\,
$f^{\,k\:\!n}(Y_i)=f^{\,i-1}f^{\,k(n-1)}f^{\,k+1-i\,}(Y_i)\subset
f^{\,i-1}f^{\,k(n-1)}(Y_1)=f^{\,i-1}(X_1)=X_i$\;\!, \,we have
\,$Y_i\subset U_i$\,, \;and hence \,$Y\!\subset U$. \,Write
$X_{k+1}=X_1$\, and \,$U_{k+1}=U_1$\;\!. \,Since $f(U_i)$ \,is a
connected set containing \,$X_{i+1}$\, and $f(U_i)\subset
f\big(O_-(X_{i\,},f^{\,k})\big)\subset
O_-(X_{i+1\,},f^{\,k})$\;\!, \,we have $f(U_i)\subset
U_{i+1}$\;\!.\, Since\, $X_1\,,\,\cdots,\,X_{k}$ \,are pairwise
disjoint, \,the main absorbed sets\, $U_1\,,\,\cdots,\,U_{k}$
\,are also pairwise disjoint. Thus \,$U$ is a component-cyclic
$f$-invariant open set with \,$k$\;\! connected components\,
$U_1\,,\,\cdots,\,U_k$\;\!. \,From (2.3) and (2.2) it is easy to
check that\, $U\!=\,\bigcup_{\;\!i=1}^{\,k}U_i$ \,is just the main
absorbed set \;${\rm{Ab}}(X,f)$\;\!.

\vspace{2mm}(3) \ For any \,$i\in\N_k$\;\!, \,from Theorem \ref{eventual invariant} we
see that $X_i$ \;\!contains at least one circle \,$C_i$\;\!,\, and
$f^{\,k}|X_i$\, has a unique minimal set $M_i$\;\!,\, which
satisfy\;
$\Omega(f^{\,k}|X_i)=\omega(f^{\,k}|X_i)=AP(f^{\,k}|X_i)=M_i\subset
C_i$\,. \,It is well known that \,$\omega(f)=\,\omega(f^{\,k})$.
\;By (2) of this theorem we have\;
$\Omega(f)\cap_{\;\!}U_i=\,\Omega(f^{\,k})\cap_{\;\!}U_i$\;\!.
\,Hence, \,from (4) and (5) of Lemma \ref{connect closed} we get $\,\omega(f)\cap
U_i=\,\omega(f^{\,k})\cap U_i=\,\omega(f^{\,k}|X_i)=M_i$ \;and\;
$\Omega(f)\cap\;\!U_i=\Omega(f^{\,k})\cap\;\!U_i=\,\Omega(f^{\,k})\cap
X_i\subset\,\Omega(f^{\,k}|X_i)\cup\;\!\partial_GX_i=M_i\cup\;\!\partial_GX_i$\;\!,
\,and from $\;\omega(f^{\,k})\cap U_i=M_i$ \;\!we see that
$M_i$\;\! is also a unique minimal set of\:\! $f^{\,k}|Y_i$ \;and
of\:\! $f^{\,k}|U_i$\,.

\vspace{2mm}(4) \,follows from (2) and (3) of this
theorem and the fact that $AP(\varphi)=AP(\varphi^{\,n})$ for any
\,$n\in\N$ \,and any continuous map\, $\varphi$ \,from a
topological space to itself (\cite{ES}).

\vspace{2mm}(5) \ For any \,$i\in\N_k$\;\!, \,from \,$X_i\subset
G-EP(f)$\, we get\, $U_{i}={\rm Ab}(X_{i\;\!},f^{\,k})\subset
G-EP(f)$\;\!. \,By (1) of Lemma \ref{connect closed} and (2) of Lemma \ref{absorb boundary},\, we
have $\,\partial_GU_i\subset EP(f)$\;\!. Thus $U_i$ is just a
connected component of\, $G-EP(f)$\;\!. \,Since $\,U_i$ is open,
\,it is also a connected component of\; $G-\overline{EP(f)}$\;\!.
\,By Corollary \ref{boundary number1},\, we have $\,|\partial_GU_i|\le\xi(U_i)$\;\!.
\,So
$\;\!|\bigcup_{j\;\!=1}^{\;k\,}\partial_GU_{\!j\,}|\le\sum_{j\;\!=1}^{\;k\,}\xi(U_{\!j})$
$\le\xi(G)$\;\!. \,From $f(U_i)\subset U_{i+1\;}$ and
$f(\partial_GU_i)\subset f\big(EP(f)\big)\subset EP(f)\subset
G-U_{i+1}$ \,we get $f(\partial_GU_i)\subset
\partial_GU_{i+1}$\;\!. Thus\,
$\bigcup_{j\;\!=1}^{\;k\,}\partial_GU_{\!j}$\, is $f$-invariant,
\,and \,$\partial_GU_i\subset
\bigcup_{j\;\!=1}^{\;k\,}\partial_GU_{\!j}\subset
EP_{\,\xi(G)}(f)$\;\!. \,Hence \,$U_i$ is also a connected
component of\;\! $G-EP_{\,\xi(G)}(f)$\;\!.
\medskip

All together, we complete the proof.
\end{proof}

\begin{cor}\label{omega equal} {\it Let $f:G\to G$ be a
graph map\;\!, \,and \,$Y$\! and \,$Y'$\! be component-cyclic
$f$-invariant closed sets contained in\, $G-P(f)$\;\!. \,If
\;$Y\cap\,Y'\ne\,\emptyset\;\!$, \,then
\;$\omega(f|Y)=\omega(f|Y')$}\;\!.
\end{cor}

\begin{proof} By (4) of Proposition \ref{sub main}, $f|Y$ and
$f|Y'$ have unique minimal sets\, $M$ and $M\:\!'$,
respectively,\, which satisfy \,$\omega(f|Y)=M$ \,and\,
$\omega(f|Y')=M\:\!'$. \,Obviously,\, both $M$ and $M\:\!'$ are
minimal sets of $f|(\:\!Y\cup\,Y')$\;\!.\, Since
$Y\cap\,Y'\!\ne\,\emptyset$\;\!, \,by Lemma \ref{union component}\;\!, \,$Y\cup\,Y'$
is also a component-cyclic $f$-invariant closed set. \,Hence, by
(4) of Proposition \ref{sub main}\;\!, \;$f|(\:\!Y\cup\,Y')$ \;\!has only one
minimal set, \,which must be $M=M\:\!'$\;\!. Thus
\;$\omega(f|Y)=\omega(f|Y')$\;\!.
\end{proof}

\vspace{3mm}The following theorem is a main result of this paper,
which describes the dynamical behavior of $f$ on the intersection of each connected
component of $G-\overline{EP(f)}$ with
$R(f)$\;\!.

\begin{thm}\label{main theorem} {\it Let  $G$  be a connected graph,  and  $f:G\rightarrow G$ be a continuous map such that $P(f)\not=\emptyset$ and $R(f)-\overline{P(f)}\not=\emptyset$.  Then there exist pairwise disjoint nonempty open subsets  $U_1, \cdots,  U_n$  of  $G$  with  $n\in \N$  such that

\vspace{0.7mm}$(1)$\ \ $f(U_i)\subset U_i,$ for each $i\in\N_n$.

\vspace{1mm}$(2)$\ \ Write $U=\bigcup_{i=1}^{\;n}\,U_i$ and $U_0=G-U$. Then $\overline{P(f)}\subset U_0$, $\overline{U}-U\subset EP(f)$,
and $\Omega(f)-U_0=R(f)-U_0=R(f)-\overline{P(f)}\subset U$.

\vspace{1.2mm}$(3)$\ \ For each \,$i\in\N_n$\;\!, $U_i$ has $k_i$ connected components $U_{i1}, \cdots, U_{ik_i}$
with $k_i\in \N$, which satisfy $f(U_{ik_i})\subset U_{i1}$ and $f(U_{ij})\subset U_{i\,,\, j+1}$ for $1\leq j<k_i$.

\vspace{1.2mm}$(4)$\ \ For each \,$i\in\N_n$\;\!, write $W_i=R(f)\cap U_i$. Then $W_i$ is a unique minimal set of $f$
contained in $U_i$.

\vspace{1.2mm}$(5)$\ \ For each \,$i\in\N_n$\;\! and $j\in N_{k_i}$, write $W_{ij}=W_i\cap U_{ij}$. Then $W_{ij}$
is a unique minimal set of $f^{k_i}$ contained in $U_{ij}$, and there is a connected closed subset $G_{ij}$ of
$G$ and a circle $C_{ij}$ such that $W_{ij}\subset C_{ij}\subset G_{ij}\subset U_{ij}$, $f(W_{ik_i})=W_{i1}$,
$f(G_{ik_i})=G_{i1}$, and $f(W_{ij})=W_{i\,,\,j+1}$ and $f(G_{ij})=G_{i\,,\,j+1}$ for $1\leq j<k_i$.

\vspace{1.2mm}$(6)$\ \ For each $i\in\N_n$, $j\in \N_{k_i}$, and for each $x\in U_{ij}$, one has
$\lim_{m\rightarrow\infty}d(f^m(x), W_i)=0$, and $\lim_{m\rightarrow\infty}d(f^{mk_i}(x), W_{ij})=0$.}
\end{thm}

\begin{proof} Since
\,$R(f)-\overline{P(f)}\ne\,\emptyset$\;\!,\,\vspace{0.3mm} by
Lemma \ref{component exist}, \,there exists at least one component-cyclic
$f$-invariant closed set \,$Y^{(1)}$\;\! in \,$G-{EP(f)}$\;\!.
\,By Proposition \ref{sub main}, \;$Y^{(1)}\subset G-\overline{EP(f)}$\;\!,
\;and\, if \;$Y^{(1)}$\;\! has \,$k_1$\, connected components then
\,$Y^{(1)}$\;\! contains at least \,$k_1$\, pairwise disjoint
circles.\,\vspace{0.3mm} Thus we can assume that there exist
\,$n$\, pairwise disjoint \;\!component-cyclic $f$-invariant
closed sets \,$Y^{(1)},\,\cdots,\,Y^{(n)}$ \;\!in
\,$G-\overline{EP(f)}$ \;with \,$n\in\N_m$\; but
\,$G-\overline{EP(f)}$ \,cannot admit \,$n+1$ \,pairwise disjoint
\;\!component-cyclic $f$-invariant closed sets, where $m$ is the maximal number of pairwise disjoint
circles in $G$. \,Suppose that
\,$Y^{(i)}$\, has \,$k_i$\, connected components.\, Then we have
\,$n\le\sum_{\;\!i=1}^{\;n}k_i\le m$\;\!.

\vspace{2mm}If
\,$R(f)-\overline{P(f)}\,\not\subset\,\bigcup_{\;\!i=1}^{\:n}Y^{(i)}$\;\!,
\,then there is a point \,$x\in
R(f)-\overline{P(f)}-\bigcup_{\;\!i=1}^{\;n}Y^{(i)}$\;\!. \,By
Lemma \ref{component exist} and Proposition \ref{sub main},\vspace{0.3mm} \,there exists a
component-cyclic $f$-invariant closed set \,$Y$ such that\, $x\in
Y\subset G-\overline{EP(f)}$\;\!. \,For each $i\in\N_n$\;\!,
\,since \;$\omega(f|Y)-\omega(f|Y^{(i)})\supset R(f|Y)-Y^{(i)}$
$\supset\{x\}\ne\,\emptyset\;\!$, \,by Corollary \ref{omega equal} \,we have
$\,Y\cap Y^{(i)}=\,\emptyset\,$.\vspace{0.3mm} So
\,$Y\:\!,\,Y^{(1)},\,\cdots,\,Y^{(n)}$ are \,$n+1$\, pairwise
disjoint component-cyclic strongly $f$-invariant closed sets in
\,$G-\overline{EP(f)}$\;\!.\vspace{0.3mm} But this will leads to a
contradiction. Hence we must have
\,$R(f)-\overline{P(f)}\,\subset\,\bigcup_{\;\!i=1}^{\:n}Y^{(i)}$\;\!.

\vspace{2mm} For each $i\in \N_n$, let $X^{(i)}$ be the component-cyclic \,strongly\,
$f$-invariant closed set contained in $Y_i$ and let
$G_{i1},\cdots, G_{ik_i}$ be the connected components of $X^{(i)}$ (see $(1)$ of Proposition \ref{sub main}).
Set $U_{ij}={\rm Ab}(G_{ij},f^{\,k_i})$, $U_i=\bigcup_{j=1}^{k_i} U_{ij}$, and $U=\bigcup_{i=1}^nU_i$.
Then, by $(2)$ of Proposition \ref{sub main} and by Lemma \ref{absorb boundary}, each $U_i$ is an $f$-invariant open set,
$f(U_{ik_i})\subset U_{i1}$ and $f(U_{ij})\subset U_{i\,,\, j+1}$ for $1\leq j<k_i$, and
${\overline U_i}-U_i\subset\bigcup_{j=1}^{k_i}({\overline U_{ij}}-U_{ij})\subset EP(f^{k_i})=EP(f)$.
So, ${\overline U}-U\subset \bigcup_{i=1}^n({\overline U_i}-U_i)\subset EP(f)$. From the definition of
$U_{ij}$, we see that $U_{ij}\cap P(f)=\emptyset$, which means that $P(f)\cap U=\emptyset$. Thus
$\overline {P(f)}\subset G-U$. Let $U_0=G-U$. Then $R(f)-U_0\subset R(f)-\overline{P(f)}$. Since
\,$R(f)-\overline{P(f)}\,\subset\,\bigcup_{\;\!i=1}^{\:n}Y^{(i)}\subset U=G-U_0$\;\!, we have
$R(f)-U_0\supset R(f)-\overline{P(f)}$. So, $R(f)-U_0=R(f)-\overline{P(f)}$. Thus $(1)$, $(2)$, and $(3)$
are proved except for the relation $\Omega(f)-U_0=R(f)-U_0$.
\medskip

As in the statement of the theorem, for each \,$i\in\N_n$\;\! and $j\in N_{k_i}$, let $W_i=R(f)\cap U_i$ and $W_{ij}=W_i\cap U_{ij}$.
From $(3)$ and $(4)$ of Proposition \ref{sub main}, we have $W_i$ and $W_{ij}$ are the unique minimal sets of $f|U_i$ and $f^{k_i}|U_{ij}$ respectively,
and there exist circles $C_{ij}$ with $W_{ij}\subset C_{ij}\subset G_{ij}\subset U_{ij}$. This together with $(2)$ and the strong
invariance of $X^{(i)}$ implies that
$f(W_{ik_i})=W_{i1}$, $f(G_{ik_i})=G_{i1}$, and $f(W_{ij})=W_{i\,,\,j+1}$ and $f(G_{ij})=G_{i\,,\,j+1}$ for $1\leq j<k_i$.
Thus $(4)$ and $(5)$ are proved. The conclusions of $(6)$ follow from $(4)$ of Theorem \ref{eventual invariant}, which clearly implies the
equation $\Omega(f)-U_0=R(f)-U_0=\bigcup_{i=1}^n W_i$. Thus the proof of $(2)$ is complete.
\end{proof}

%%%%%%%%%%%%%%%%%%%%%%%%%%%%%%%%%%%%%%%%%%%%%%%%%%%%%%%%%%%%%%%%%%
\section{Applications of the main theorem}

As applications of Theorem \ref{main theorem}, we will prove several propositions part of which improve or reprove
some known results.

\medskip
The following theorem  is also implied by \cite[Theorem 4]{Bl86}.

\begin{thm}\label{r structure} {\it Let $f:G\to G$ be a
graph map\;\!, \,and let \,$m$ be the greatest number of\;\!
pairwise disjoint circles in \;\!$G$. \vspace{0.5mm} \,Then there
exist minimal sets $M_1\:\!,\,\cdots,\,M_n$ of $f$ in
\,$G-\overline{EP(f)}$ \,with \;$0\le n\le m$ \,such that
\;$\overline{R(f)}=\,\overline{P(f)}\cup\big(\bigcup_{\,i=1}^{\;n}M_i\big)$}\;\!.
\end{thm}

\begin{proof} \,By \cite[Theorem 2.1]{MS2},\, we
get
\,$\overline{R(f)}=R(f)\cup\overline{P(f)}=\overline{P(f)}\cup\big(R(f)-\overline{P(f)}_{\,}\big)$\;\!.
If\, $R(f)-\overline{P(f)}=\,\emptyset$\, then we can put
\,$n=0$\;\!.\, Otherwise,
$R(f)-\overline{P(f)}\ne\,\emptyset$\;\!. \,Let \,$n$\,, $U_i$ and the minimal sets \,$W_i$ \,be the same as in
Theorem \ref{main theorem}. Then, by Theorem \ref{main theorem}, we have $1\le n\le m$\;\! and
$$
\mbox{$R(f)-\overline{P(f)}=\;\!\bigcup_{\,i=1}^{\;n}\big(R(f)\cap
U_i\big)=\;\!\bigcup_{\,i=1}^{\;n}W_i\subset
G-\overline{EP(f)}\,.$} \vspace{-2mm}
$$
Hence, \,writing $M_{i\,}$ for $W_i$, \,we obtain
\;$\overline{R(f)}=\,\overline{P(f)}\cup\big(\bigcup_{\,i=1}^{\;n}M_{i\,}\big)$\;\!.
 \end{proof}

%\begin{rem} Recall that the closure
%$\overline{R(f)}$ of the set of all recurrent points of a
%continuous map $f: X\to X$ is called the {\it{center}}
%\;\!of $f$\,(cf. \cite[p.\;77 ]{BC}), which is written\,
%$C(f)$\;\!. \,In \cite[Theorem 4]{Bl86}, \,Blokh showed that, for
%any graph map $f: G\to G$,  \vspace{-2mm}
%$$
%\overline{P(f)}={\mbox{$\bigcup_{\,i=1}^{\;b}$}}B(M_i)\,\cup\,(\,\mbox{$\bigcup_{\,\alpha\in
%A}$}\,\Omega\;\!'_{\alpha})\,\cup\,X_f\,,  \vspace{-2mm}
%$$
%$$
%C(f)={\mbox{$\bigcup_{\,i=1}^{\;b}$}}B(M_i)\,\cup\,(\,\mbox{$\bigcup_{\,\alpha\in
%A}$}\,\Omega\;\!'_{\alpha})\,\cup\,(\,\mbox{$\bigcup_{\,i=1}^{\;s}$}S(\mbox{\it\~{M}}_i))\,\cup\,X_f\,.
%\vspace{-0mm}
%$$
%Thus the above Theorem \ref{r structure} \;\!can be also derived from
%\cite[Theorem 4]{Bl86}\;\!.
%\end{rem}
The following proposition indicates that, for a graph map $f:G\to G$ and
\,$x\in G$, \,if every neighborhood of \,$x$ contains both a
recurrent point and an eventually periodic point of $f$, then every neighborhood of \,$x$ must contain a periodic point.

\begin{prop} \,{\it Let $f$ be a graph map. Then
$\overline{R(f)}\;\!\cap\,\overline{EP(f)}\,=\,\overline{P(f)}$.}
\end{prop}

\begin{proof} \,By \cite[Theorem
2.1]{MS2}\,\vspace{0.5mm} we get
\,$\overline{R(f)}=R(f)\cup\overline{P(f)}=\overline{P(f)}\cup\big(R(f)-\overline{P(f)}_{\,}\big)$\;\!,
\,and by Theorem \ref{main theorem} \,we get
\,$R(f)-\overline{P(f)}\,\subset G-\overline{EP(f)}$\;\!. \,Hence
\vspace{-1mm}\,
$$
\big(\;\!\overline{R(f)}\;\!\cap\;\overline{EP(f)}\;\big)\,-\,\overline{P(f)}
=\big(\;\!\overline{R(f)}\,-\,\overline{P(f)}\;\big)\;\!\cap\;\overline{EP(f)}
=\big({R(f)}\,-\,\overline{P(f)}\;\big)\;\!\cap\;\overline{EP(f)}\,=\;\emptyset\,.
\vspace{-1mm}
$$
In addition, \,$\overline{R(f)}\cap\;\!\overline{EP(f)}\,\supset
\,\overline{P(f)}$\,follows from
${R(f)}\,\supset{P(f)}$ \,and\, ${EP(f)}\,\supset{P(f)}$.\, Thus we have\,
$\overline{R(f)}\;\!\cap\,\overline{EP(f)}\,=\,\overline{P(f)}$\;\!.
\end{proof}

\begin{example}\label{special map} For special graph maps,
\,by Theorem \ref{main theorem} or Theorem \ref{r structure} we can obtain some further
detailed information. For example, let $S^{\:\!1}$ be the unit
circle in the complex plane $\C$. For $n\in\N$\;\!, \,let
$T_n=\{z\in\C:\,z^{\:\!n}\in[\:\!0,2^{\:\!n}]\:\!\}$ \,and let\,
$G_n=S^{\:\!1}\!\cup T_n$\;\!. Then any two circles in $G_n$
intersect. Hence,\vspace{0.3mm} \,from Theorem \ref{main theorem} we see that,
\,for any $f\in C^{\;\!0}(G_n)$\;\!, \;if
\,$\overline{R(f)}\ne\overline{P(f)}$\, then there exist a unique
connected component \,$U$ of \,$G-\overline{EP(f)}$\;\!, \,a
unique strongly $f$-invariant connected closed set $X$,\, a
circle\, $C$ in \,$G$ and a unique minimal set $M$\, of $f$ such
that \,$U\cap\,R(f)\ne\,\emptyset$\, and\, $M\subset C\subset
X\subset U$. \,By Theorem \ref{r structure}, \,we have
\,$\overline{R(f)}=\overline{P(f)}\cup M$. \,By Theorem \ref{eventual invariant}, \,this
minimal set $M$\, is totally minimal.
\end{example}

Noting that
$AP(f)=\bigcup\:\{M:M$ is a minimal set of $f{\,}\}$ and
$P(f)\subset AP(f)\subset R(f)_{\;\!}$, by Theorem \ref{r structure}, we
get \,$\overline{R(f)}\subset\!
AP(f)_{\:\!}\cup_{\:\!}\overline{P(f)}\subset\overline{R(f)}$\;\!. Then we have

\begin{thm}\label{R closure}{\it Let $f:G\to G$ be a
graph map\;\!. Then
\,$\overline{R(f)}=AP(f)\cup\overline{P(f)}$}\;\!.
\end{thm}

For any graph map $f$, Hawete showed that
\,$\overline{R(f)}=\overline{AP(f)}$ (see \cite[Lemma 3.1]{Haw})\;\!.
\,In general, if $X$ is topological space and  $f\in C^{\;\!0}(X)$\;\!, \, then from the relation
\,$P(f)\subset AP(f)\subset R(f)$ \,we can easily get that the condition
\,$\overline{R(f)}=AP(f)\cup\overline{P(f)}$ \,implies
\,$\overline{R(f)}=\overline{AP(f)}$ \,and
\,$\overline{R(f)}=R(f)\cup\overline{P(f)}$\;\!. The following examples show that neither
\,$\overline{R(f)}=\overline{AP(f)}$ \,nor
\,$\overline{R(f)}=R(f)\cup\overline{P(f)}$ \,implies
\,$\overline{R(f)}=AP(f)\cup\overline{P(f)}$\;\!. Thus Theorem
\ref{R closure}\;\! is an essential improvement of\;\! \cite[Lemma 3.1]{Haw}
and \cite[Theorem 2.1]{MS2}.

\begin{example} (1)\ \ In \cite[Example
3.3]{MS3} the authors constructed an isometric homeomorphism $f$
from a complex Hilbert space $X$ to itself, \,which satisfies
$R(f)=X$ and \;\!$AP(f)=\,\emptyset$\;\!. \,For this $f$, \,we
have \,$\overline{R(f)}=R(f)\cup\overline{P(f)}$ \,but have
neither \,$\overline{R(f)}=AP(f)\cup\overline{P(f)}$ \,nor
\,$\overline{R(f)}=\overline{AP(f)}$\;\!.

\vspace{2mm}(2)\ \ Let \,$g:[\;\!0,1]\to[\;\!0,1]$ \,be the {\it
tent map}\, defined by \,$g(x)=\min\{2\:\!x,2-2\:\!x\}$ \,for all
\,$x\in[\;\!0,1]$\;\!, \,and let \,$h:S^1\to S^1$ \,be an
irrational rotation. Put $X=[\;\!0,1]\times S^1$. Then $X$ is a
cylinder.\, Define $f:X\to X$ by $f(x,y)=\big(g(x),h(y)\big)$
\,for any \,$(x,y)\in X$. Then \,$P(f)=\,\emptyset$\;\!, \,and
both \,${AP(f)}$ \,and \,${X-R(f)}$ \,are dense subsets of $X$.
\vspace{0.3mm} \,For this $f$, \,we have
\,$\overline{R(f)}=\overline{AP(f)}$ \,but have neither
\,$\overline{R(f)}=AP(f)\cup\overline{P(f)}$ \,nor
\,$\overline{R(f)}=R(f)\cup\overline{P(f)}$\;\!.
\end{example}

From \cite[Theorem 2.1]{MS2}\,, \vspace{0.6mm}we know that
\,$\overline{R(f)}-\overline{P(f)}=\big(R(f)\cup\overline{P(f)}_{\;}\big)-\overline{P(f)}=R(f)-\overline{P(f)}$\;\!.
\,By Theorem \ref{r structure}, \
$\overline{R(f)}-\overline{P(f)}\,=\,\bigcup_{\,i\,=\,1}^{\;n}M_i$ which is closed.
Hence we have

\begin{prop}\label{positive distance} \,\ {\it Let $f:G\to G$
be a graph map\;\!. \vspace{0.3mm}If \;${P(f)}\ne\,\emptyset$\,
and\, \,$R(f)-\overline{P(f)}\ne\,\emptyset$\,,\,\, then the
distance \;$d\big(\,{R(f)}-\overline{P(f)}\ ,\
\overline{P(f)\;\!\!}\;\big)\,=\,d\big(\,\overline{R(f)}-\overline{P(f)}\
,\ \overline{P(f)\;\!\!}\;\big)>0$}\;\!.
\end{prop}

The following theorem is first given by Zhang, Liu and
Qin in \cite{ZLQ}, \,which can be also obtained as a corollary
of\;\! Proposition \ref{positive distance}\;\!.

\begin{thm}\label{dense periodic}(\cite[Theorem 2.8]{ZLQ})\ \,\ {\it Let $f:G\to G$ be a graph map\;\!. If
\;$\overline{R(f)}=\,G$ and \;$P(f)\ne\,\emptyset$\;\!, \;then\,
$\overline{P(f)}=\,G$}\;\!.
\end{thm}

\begin{proof} Assume to the contrary that\, $\overline{P(f)}\ne G$.
\,Then,\, by Proposition \ref{positive distance}\;\!,\, \vspace{0.5mm}there will be a
neighborhood \,$U$ of \,$\overline{P(f)}$ \;in $G$ \,such that
\,$U-\;\!\overline{P(f)}\ne\,\emptyset$ \,and
\,$U\cap\,\big(\;\!{R(f)}-
\;\!\overline{P(f)}\;\big)=\,\emptyset$\;\!.\vspace{0.5mm} \,But
these will imply that \,$\overline{R(f)}\ne G$. This is a contradiction.
\end{proof}

By \;\!Theorem \ref{main theorem}, \,for any \;\!$x\in
\overline{R(f)}-\overline{P(f)}={R(f)}-\overline{P(f)}$\;\!,
\vspace{0.3mm} \,there exist a component-cyclic strongly
$f$-invariant closed set \,$X_i\subset
G-\overline{EP(f)}\subset G-\overline{P(f)}$ \,and a circle
\,$C_{ij}\subset X_i$ \,such that
\,$x\in\;\!C_{ij}$. \;\!Then,  we
 obtain the following

\begin{prop}\label{r=p} {\it Let $f:G\to G$ be
a graph map \vspace{0.3mm} \,and \,$W$ be a subset of
\,$G-\overline{P(f)}$\;\!. \,If\; $W\cap\;\!C=\,\emptyset$ \,for
any circle \;$C\subset G-\overline{P(f)}$\;\!,\, then\;
$W\cap\,\overline{R(f)}=\,\emptyset$\;\!. \vspace{0.3mm}
\;Specially, \;if \;$W$ is a connected component of
\;$G-\overline{P(f)}$\, which contains no circle,\, \vspace{0.4mm}
then \;$W\!\cap\,\overline{R(f)}=\,\emptyset\;;$ \;if
\;$G-\overline{P(f)}$\, contains no circle,\; then
\,$\overline{R(f)}=\overline{P(f)}$}\,.
\end{prop}

From\;\! Proposition \ref{r=p}\;\! we can directly derive
the following theorem, \,which is first given by Ye\;\! in\;\!
\cite{Ye}\;\!.

\begin{thm}\;(\cite[Theorem
2.6]{Ye})\ \ {\it Let $f:T\to T$ be a tree map. \,Then
\,$\overline{R(f)}=\overline{P(f)}$}\;\!.
\end{thm}

\begin{prop}\label{intersect boundary} {\it Let $f:G\to G$ be
a graph map\;\!,\, and \,$A=[x,y]$ be an arc in\; $G$.

\vspace{0.5mm}$(1)$\ \ If \,$(x,y)_A\;\!\cap\,{\rm
Br}(G)=\,\emptyset$\, and\,
$A\;\!\cap\,\overline{P(f)}\;\!\ne\,\emptyset$\,,\, then
\,$(x,y)_A\;\!\cap\;\!\big(\;\!\overline{R(f)}-\overline{P(f)}\,\big)=\,\emptyset$\;\!.

\vspace{0.5mm}$(2)$\ \ If\; $x\in\overline{P(f)}$\, and\;
$y\in\overline{R(f)}-\overline{P(f)}$\,, \,then\; $(x,y]_A\cap
\;\!{\rm Br}(G)\ne\,\emptyset$\;\!.}
\end{prop}

\begin{proof} \ (1)\ \ Under the given
conditions,\, for any circle \,$C\!\subset
G-\overline{P(f)}$\;\!,\, we have
$(x,y)_A\;\!\cap\;C=\,\emptyset$\;\!. \,Let
$W=(x,y)_A-\overline{P(f)}$\;\!. \,By Proposition \ref{r=p}\,,\, we
have $W\cap\,\overline{R(f)}=\,\emptyset$\;\!. \,This means that
\,$(x,y)_A\;\!\cap\;\!\big(\;\!\overline{R(f)}-\overline{P(f)}\,\big)=\,\emptyset$\;\!.

\vspace{2mm}(2)\ \ If\, $x\in\overline{P(f)}$\; and\;\!
$(x,y]_A\;\!\cap\:{\rm Br}(G)=\,\emptyset$\;\!, \,then\, $y\notin
C$ for any circle \;$C\subset G-\overline{P(f)}$\;\!,\, and it
follows from Proposition \ref{r=p} that\,
$y\notin\overline{R(f)}-\overline{P(f)}$\;\!. Therefore, \,if\,
$x\in\overline{P(f)}$\; and\,
$y\in\overline{R(f)}-\overline{P(f)}$\; then we must have
\,$(x,y]_A\cap\;\!{\rm Br}(G)\ne\,\emptyset$\;\!.
 \end{proof}

\begin{example} \ Let\,
$G_n=\;\!S^{\:\!1}\!\cup\;\!T_n\subset\;\!\Bbb C$ \,be the same as
in Example \ref{special map}\;\!, \,and let $f\in C^{\;\!0}(G_n)$\;\!. \,If the
origin \,$0\,\in\,\overline{P(f)}$\;\!, \,and
\,$S^{\:\!1}\;\!\cap\;\overline{P(f)}\;\!\ne\,\emptyset$\;\!,
\,then \;$G_n-\overline{P(f)}$\, contains no circle, \,and from
Proposition \ref{r=p} \;\!we get
\,$\overline{R(f)}=\overline{P(f)}$\;\!.
\end{example}

\begin{thm}\label{irrational rotation}\,{\it Let $f:G\to G$ be a
graph map\;\!,\, and \;$U$ be a connected component of
\;$G-\overline{EP(f)}$\, with \;$U\cap\,R(f)\ne\,\emptyset$\,.
Then there exists \,$k\in\N$ \,such that\, $f^{\;\!k}(U)\subset
\;\!U$, \,and\;\! $f^{\;\!k\;\!}|\;\!U$ is topologically
semi-conjugate to an irrational rotation of\;\! the unit circle
$S^1$}\;\!.
\end{thm}

\begin{proof} Use the all notations in Theorem
\ref{main theorem}. From Theorem \ref{main theorem} we see that there exist \,$i\in\N_n$ \,and
\;\!$j\in\N_{k_i}$ \,such that\vspace{0.3mm}
\,$U=\,U_{ij}\!={\rm Ab}(X_{ij}\!,f^{\;\!k_i})$\;\!.
\,Let\;\! $k=k_i$\;\! and write\;\! $X\!=X_{ij}$.\vspace{0.5mm}
Then $f^{\;\!k}(X)=X$ and $f^{\;\!k}(U)\subset U$.\, By Theorem
\ref{eventual invariant}\;\!,\, $f^{\;\!k\;\!}|X$ is topologically semi-conjugate to some
irrational rotation\;\! $h:S^1\to S^1$, \,that is, \,there is a
continuous surjection\, $\varphi:X\to S^1$ such that\,
$h_{\,}{\varphi}={\varphi}_{\;\!}f^{\;\!k\;\!}|X$\;\!. \,By means
of\;\! ${\varphi}$ we define a map\, $\psi:U\to S^1$ as follows.
\,For any \,$x\in U$, \,taking an \,$n\in\Z_+$ \,such that
$f^{\;\!kn}(x)\in X$\;\!, \,and then we put\,
$\psi(x)=h^{-n}\;\!{\varphi}\;\!f^{\;\!kn}(x)$\;\!. \,Write
\,$x_n=f^{\;\!kn}(x)$\;\!. \,For any \,$i\in\N$\;\!, \,we have \
$h^{-n-i}\;\!{\varphi}\;\!f^{\;\!kn+ki}(x)\;\!=\,h^{-n-i}\;\!{\varphi}\;\!f^{\;\!ki}(x_n)
\;\!=\,h^{-n-i}\;\!h^{\,i}\;\!{\varphi}(x_n)\;\!=\,h^{-n}\;\!{\varphi}\;\!f^{\;\!kn}(x)$\;\!.
This means that the definition of\, $\psi(x)$ \,is independent of
the choice of \;\!$n$\;\!. \,So we obtain a map\, $\psi:U\to S^1$.
This\;\! $\psi$ \;\!is surjective, \,since\,
${\varphi}\!=\psi_{\;\!}|_{\;\!}X$ \;\!is surjective. \,By (3) of
\;\!Lemma \ref{connect closed}, \,for any\, $x\in U$, \,there exist an \,$m\in\N$\,
and an open neighborhood\, $Z$ \;\!of\, $x$ \;\!in $U$ such that
$f^{\;\!km}(Z)\subset X$\;\!, \,which implies that\,
$\psi_{\;\!}|Z=h^{-m}\;\!{\varphi}\;\!f^{\;\!km\;\!}|Z$ \,is
continuous. Thus\;\! $\psi$ \;\!is continuous. \,From\,
$h_{\;\!}\psi(x)=h\;\!h^{-n}\;\!{\varphi}\;\!f^{\;\!kn}(x)=
h^{-n+1}\;\!{\varphi}\;\!f^{\;\!kn-k}(f^k(x))=\,\psi f^{\;\!k}(x)$
we get\, $h_{\;\!}\psi=\,\psi f^{\;\!k\;\!}|\;\!U$. \,Hence
$f^{\;\!k\;\!}|\;\!U$ \,is topologically semi-conjugate to the
irrational rotation\;\! $h:S^1\to S^1$. Thus the theorem \;\!is
proven. \end{proof}

\begin{rem} \ Theorem \ref{irrational rotation}\, and $(5)$ of Theorem
\ref{eventual invariant} \,seem  similarly, \,but there are some differences between them. \,In $(5)$ of Theorem \ref{eventual invariant},
\,since the graph map $f:G\to G$ has no periodic point, \,there
exist a unique strongly $f$-invariant compact set $X$ and an\,
$n\in\Bbb N\cup\{0\}$\, such that $f^{\,n}(G)\subset X$.
\,However, \,in Theorem \ref{irrational rotation}, \,if $f$\, has periodic points,
\,then the $f^{\;\!k}$-invariant set $U$ is open in $G$ and is not
compact, \,and there is no\;\! $n\in\Bbb N$ \;\!such that
$f^{\,kn}(U)$ is contained in the strongly $f^{\;\!k}$-invariant
compact set $X\!=X_{ij}$.
\end{rem}

{\noindent\bf Acknowledgement.} {Jiehua Mai, Kesong Yan, and Fanping Zeng  are supported by NNSF of China (Grant No.
12261006) and  NSF of Guangxi Province (Grant No.
2018GXNSFFA281008);  Kesong Yan is also supported by NNSF of China
(Grant No. 12171175) and Project of Guangxi First Class Disciplines
of Statistics and Guangxi Key Laboratory of Quantity Economics; Enhui Shi is supported by NNSF of China (Grant No. 12271388).}

\end{document}